\documentclass[reqno]{amsart}
\usepackage{amssymb}

\newtheorem{theorem}{Theorem}[section]
\newtheorem{proposition}[theorem]{Proposition}
\newtheorem{lemma}[theorem]{Lemma}

\newtheorem{cor}[theorem]{Corollary}
\theoremstyle{definition}
\newtheorem{definition}[theorem]{Definition}

\allowdisplaybreaks
\numberwithin{equation}{section}
\begin{document}

\title[The trace of Frobenius of elliptic curves]
{$p$-Adic hypergeometric functions and the trace of Frobenius of elliptic curves}


 \author{Sulakashna}
\address{Department of Mathematics, Indian Institute of Technology Guwahati, North Guwahati, Guwahati-781039, Assam, INDIA}
\curraddr{}
\email{sulakash@iitg.ac.in}
\author{Rupam Barman}
\address{Department of Mathematics, Indian Institute of Technology Guwahati, North Guwahati, Guwahati-781039, Assam, INDIA}
\curraddr{}
\email{rupam@iitg.ac.in}

\thanks{}


\date{6th November 2023, version-1}
\begin{abstract} 
	Let $p$ be an odd prime and $q=p^r$, $r\geq 1$. For positive integers $n$, let ${_n}G_n[\cdots]_q$ denote McCarthy's $p$-adic hypergeometric functions. 
	In this article, we prove an identity expressing a ${_4}G_4[\cdots]_q$ hypergeometric function as a sum of two ${_2}G_2[\cdots]_q$ hypergeometric functions. This identity generalizes some known identities satisfied by the finite field hypergeometric functions. We also prove a transfomation that relates ${_{n+2}}G_{n+2}[\cdots]_q$ and ${_n}G_n[\cdots]_q$ hypergeometric functions.
	Next, we express the trace of Frobenius of elliptic curves in terms of special values of
	${_4}G_4[\cdots]_q$ and ${_6}G_6[\cdots]_q$ hypergeometric functions. Our results extend the recent works of Tripathi and Meher on the finite field hypergeometric functions to wider classes of primes.    
\end{abstract}
\maketitle
\section{Introduction and statement of results}
The arithmetic properties of Gauss and Jacobi sums have a very long history in number theory. 
Number theorists have obtained finite field and $p$-adic analogues of classical hypergeometric series using these sums, and these functions have significant applications in arithmetic geometry. 
In recent times, many authors have studied certain finite field analogues of the classical hypergeometric series. It seems that hypergeometric functions over a finite field first appeared in Koblitz's work \cite{kob2}. 
There are other definitions of hypergeometric functions over finite fields. For example, see the works of Greene \cite{greene, greene2}, Katz \cite{katz}, McCarthy \cite{mccarthy3}, Fuselier et al. \cite{FL}, and Otsubo \cite{noriyuki}. Hypergeometric functions over finite fields are also known as \emph{Gaussian hypergeometric series}.
Some of the biggest motivations for studying Gaussian hypergeometric series have been their connections with Fourier coefficients and eigenvalues of modular forms and with counting points on certain kinds of algebraic varieties.
For example, see \cite{ahlgren, BK, BK1, evans-mod, frechette, fuselier, Fuselier-McCarthy, koike, lennon, lennon2, mccarthy4, mc-papanikolas, mortenson, ono, salerno, vega}. However, results involving Gaussian hypergeometric series are often restricted to primes in certain congruence classes to facilitate the existence of characters of specific orders. To overcome these restrictions, McCarthy \cite{mccarthy1, mccarthy2} defined a function in terms of quotients of the $p$-adic gamma function which can best be described as an analogue of hypergeometric functions in the $p$-adic setting. In this article, we make use of these functions, as developed by Greene, McCarthy, 
and Ono \cite{greene, greene2,mccarthy2,mccarthy3,ono} and express the trace of Frobenius of elliptic curves in terms of McCarthy's $p$-adic hypergeometric functions.
\par In a recent paper \cite{TB}, the second author and Tripathi found an identity expressing a ${_4}F_3$-Gaussian hypergeometric series as a sum of two $_{2}F_1$-Gaussian hypergeometric series.
In a very recent paper \cite{TM}, Tripathi and Meher proved finite field analogues of certain classical identities that relate Appell series to classical hypergeometric series. As an application, they also found another such summation identity over finite fields. 
In the first part of this article, we prove two identities for McCarthy's $p$-adic hypergeometric functions. Our first summation identity generalizes the summation identities proved in \cite{TB, TM}.  For an odd prime $p$, let $q=p^r, r\geq 1$. 
For a positive integer $n$, let $_{n}G_{n}[\cdots]_q$ denote McCarthy's $p$-adic hypergeometric function (see Definition \ref{defin1} in Section \ref{pre}). In the following theorem, we prove a general identity expressing a ${_4}G_4[\cdots]_q$ hypergeometric
function as a sum of two ${_2}G_2[\cdots]_q$ hypergeometric functions.
\begin{theorem}\label{MT-1}
	For $k=1,\ldots, 4$, let $a_k=\frac{m_k}{d_k}$ be rational numbers such that $\gcd(m_k,d_k)=1$. Let $p$ be an odd prime such that $p\nmid d_1d_2d_3d_4$. Let $q=p^r$, $r\geq1$ such that $q\equiv 1\pmod{d}$, where $d=\emph{lcm}\{d_1, d_2, d_3, d_4\}$. Then, for $x\in \mathbb{F}_q$, we have   
	\begin{align*}
		_2G_2\left[\begin{array}{cc}
			a_1,\hspace*{-0.15cm} & a_2 \vspace{.12cm}\\
			a_3,\hspace*{-0.15cm} & a_4
		\end{array}|x \right]_q + {_2}G_2\left[\begin{array}{cc}
				a_1,\hspace*{-0.15cm} & a_2 \vspace{.12cm}\\
			a_3,\hspace*{-0.15cm} & a_4
		\end{array}|-x
		\right]_q ={_4}G_4\left[\begin{array}{cccc}
			\frac{a_1}{2},\hspace*{-0.15cm} & \frac{1+a_1}{2},\hspace*{-0.15cm} & 	\frac{a_2}{2},\hspace*{-0.15cm} & \frac{1+a_2}{2}\vspace{.12cm}\\
				\frac{a_3}{2},\hspace*{-0.15cm} & \frac{1+a_3}{2},\hspace*{-0.15cm} &\hspace*{-0.15cm} 	\frac{a_4}{2},\hspace*{-0.15cm} & \frac{1+a_4}{2}
		\end{array}|x^2
		\right]_q.
	\end{align*} 
\end{theorem}
If we substitute $a_2=1-a_1,a_3=0$, and $a_4=0$ in Theorem \ref{MT-1}, then we obtain a $p$-adic analogue of \cite[Theorem 1.3]{TM} for $q\equiv1,d_1+1\pmod{2d_1}$. If we substitute $a_3=0$ and $a_2=\frac{a_4}{2}$, 
then for $q$ satisfying $q\equiv1\pmod{d_1}$ and $q\equiv1\pmod{2d_4}$, we obtain a $p$-adic analogue of \cite[Theorem 1.8]{TB}.
\par Next, we prove a transformation between ${_{n+2}}G_{n+2}[\cdots]_p$ and ${_n}G_{n}[\cdots]_p$ hypergeometric functions. 
	\begin{theorem}\label{thrm-1}
	Let $p$ be an odd prime. For a positive integer $n$, let $a_k, b_k\in \mathbb{Q}\cap \mathbb{Z}_p$ for $k=1, \ldots, n$. Let $d$ be a positive integer. If $p\equiv-1\pmod d$ then, for $t \in \mathbb{F}_p$, we have
	\begin{align*}
		_{n+2}G_{n+2}\left[\begin{array}{ccccc}
			a_1, &  \ldots, & a_n, &\frac{1}{d}, &\frac{d-1}{d} \vspace*{0.1cm}\\
			b_1, &  \ldots, & b_n, &\frac{1}{d}, &\frac{d-1}{d}
		\end{array}|t
		\right]_p={_n}G_n\left[\begin{array}{cccc}
			a_1, & a_2, & \ldots, & a_n \vspace*{0.1cm}\\
			b_1, & b_2, & \ldots, & b_n
		\end{array}|t
		\right]_p.
	\end{align*}
\end{theorem}
\subsection{The trace of Frobenius of elliptic curves defined over $\mathbb{F}_q$} Let $p$ be an odd prime, and let $\mathbb{F}_q$ be the finite field containing $q$ elements, where $q=p^r,r\geq1$. Let $E/\mathbb{F}_q$ be an elliptic curve given in the Weierstrass form. Then the trace of
Frobenius endomorphism $a_q(E)$ of $E$ is given by
\begin{align*}
a_q(E):=q+1-\#E(\mathbb{F}_q),
\end{align*}
where $\#E(\mathbb{F}_q)$ denotes the number of $\mathbb{F}_q$-points on $E$ including the point at infinity. It is well-known that the trace of Frobenius of elliptic curves can be expressed as special values of ${_2}F_1$-Gaussian hypergeometric series, 
see for example \cite{BK,BK1,fuselier,koike, lennon,lennon2,ono}. It was Ono \cite{ono} who first expressed the trace of Frobenius of elliptic curves in terms of ${_3}F_2$-Gaussian hypergeometric series. 
Very recently, Tripathi and Meher \cite{TM} expressed the trace of Frobenius and the sum of traces of Frobenius of certain families of elliptic curves in terms of ${_4}F_3$-Gaussian hypergeometric series for primes in some congruence classes. 
\par The function $_{n}G_{n}[\cdots]_q$ often allows results involving Gaussian hypergeometric series to be extended to a wider class of primes, see for example \cite{BS2, BS1, BS3, BS4,BSM, mccarthy2,NS,NS1,SB1,SB}. 
In \cite{mccarthy2}, McCarthy expressed the trace of Frobenius of elliptic curves in terms of a special value of $_{2}G_{2}[\cdots]_p$ hypergeometric function for all primes $p>3$. Later, the second author with Saikia \cite{BS1}  
expressed the trace of Frobenius of elliptic curves in terms of another special value of $_{2}G_{2}[\cdots]_q$ hypergeometric function. However, there is no formula for the trace of Frobenius of elliptic curves in terms of special values of $_{n}G_{n}[\cdots]_q$ hypergeometric function with $n\geq 3$ to date which holds for all but finitely many primes. In this article, we find several expressions for the traces of Frobenius endomorphism of certain families of elliptic curves in terms of special values of ${_4}G_{4}[\cdots]_q$ and ${_6}G_{6}[\cdots]_q$ hypergeometric functions
which hold for all but finitely many primes. 
\par
 Let $\varphi$ denote the quadratic character on $\mathbb{F}_q$. In the following theorem we express the sum of traces of Frobenius of elliptic curves as a special value of ${_4}G_{4}[\cdots]_{p^r}$ hypergeometric function for all odd prime $p$ and $r\geq 1$. 
 This gives a $p$-adic analogue of \cite[Theorem 1.4]{TM}. We note that \cite[Theorem 1.4]{TM} holds only for $q=p^r\equiv 1\pmod{4}$.
\begin{theorem}\label{MT-3}
	Let $p$ be an odd prime and $q=p^r,r\geq 1$. Let $E_\lambda:y^2=x(x-1)(x-\lambda)$ and $E_{-\lambda}:y^2=x(x-1)(x+\lambda)$ be elliptic curves over $\mathbb{F}_q$ such that $\lambda\notin\{0,\pm 1\}$. Then we have
\begin{align*}
		a_q(E_\lambda)+a_q(E_{-\lambda})=\varphi(-1)\cdot{_4}G_4\left[\begin{array}{cccc}
			0, & \frac{1}{2}, & 0, & \frac{1}{2}\vspace*{0.1cm}\\
			 \frac{1}{4}, & \frac{3}{4}, & \frac{1}{4}, & \frac{3}{4}
		\end{array}|\lambda^2
		\right]_q.
\end{align*} 
\end{theorem}
The following theorem gives a $p$-adic analogue of \cite[Theorem 1.6]{TM}. Our result holds for all $q=p^r,r\geq1$ with $p>3$, whereas \cite[Theorem 1.6]{TM} holds only for $q=p^r\equiv 1\pmod{3}$. 
\begin{theorem}\label{MT-4}
	Let $p>3$ be a prime and $q=p^r,r\geq 1$. Let $E_{a_1,a_3}:y^2+a_1xy+a_3y=x^3$ and $E_{a_1,-a_3}:y^2+a_1xy-a_3y=x^3$ be elliptic curves over $\mathbb{F}_q$ such that $a_1,a_3\in\mathbb{F}_q^\times$. Then we have
	\begin{align*}
		a_q(E_{a_1,a_3})+a_q(E_{a_1,-a_3})={_4}G_4\left[\begin{array}{cccc}
			0, & \frac{1}{2}, & 0, & \frac{1}{2}\vspace*{0.1cm}\\
			\frac{1}{6}, & \frac{1}{3}, & \frac{2}{3}, & \frac{5}{6}
		\end{array}|\frac{729a_3^{2}}{a_1^6}
		\right]_q.
	\end{align*}
\end{theorem}
In the following theorem, we prove a $p$-adic analogue of \cite[Theorem 1.5]{TM}. Our result holds for all $q=p^r,r\geq1$ with $p$ odd, but \cite[Theorem 1.5]{TM} holds only for $q=p^r\equiv 1\pmod{8}$.
\begin{theorem}\label{MT-5}
	Let $p$ be an odd prime and $q=p^r,r\geq 1$. Let $E_{f,g}:y^2=x^3+fx^2+gx$ and $E_{f,-g}:y^2=x^3+fx^2-gx$ be elliptic curves over $\mathbb{F}_q$ such that $f,g\in\mathbb{F}_q^\times$. Then we have 
	\begin{align*}
		a_q(E_{f,g})+a_q(E_{f,-g})=\varphi(f)\cdot{_4}G_4\left[\begin{array}{cccc}
			0, & \frac{1}{2}, & 0, & \frac{1}{2}\vspace*{0.1cm}\\
			\frac{1}{8}, & \frac{3}{8}, & \frac{5}{8}, & \frac{7}{8}
		\end{array}|\frac{16g^{2}}{f^4}
		\right]_q.
	\end{align*}
\end{theorem}
In our next theorem, we express the sum of traces of Frobenius of a family of elliptic curves in terms of ${_6}G_6[\cdots]_q$ for all $q=p^r,r\geq1$, where $p>3$. 
\begin{theorem}\label{MT-6.0}
		Let $p>3$ be a prime and $q=p^r,r\geq 1$. Let $E_{c,d}:y^2=x^3+cx^2+d$ and $E_{c,-d}:y^2=x^3+cx^2-d$ be elliptic curves over $\mathbb{F}_q$ such that $c,d\in\mathbb{F}_q^\times$. Then we have
		\begin{align*}
			a_q(E_{c,d})+a_q(E_{c,-d})&=\varphi(c)\cdot{_6}G_6\left[\begin{array}{ccccccc}
				0,& \frac{1}{2}, & 0,& \frac{1}{2},& \frac{1}{4},& \frac{3}{4}\vspace*{0.05cm}\\
				\frac{1}{12},& \frac{1}{4},& \frac{5}{12},& \frac{7}{12},& \frac{3}{4},& \frac{11}{12}
			\end{array}|\frac{729d^2}{16c^6}\right]_q\\
		&\hspace{0.5cm}-\varphi(d)-\varphi(-d).
		\end{align*}
\end{theorem}
In \cite[Theorem 1.7]{TM}, the sum of traces of Frobenius of the elliptic curves $E_{c,d}:y^2=x^3+cx^2+d$ and $E_{c,-d}:y^2=x^3+cx^2-d$ is expressed as a special value of a ${_4}F_3$-Gaussian hypergeometric series under the condition that $q=p^r\equiv 1\pmod{12}$. 
In Theorem \ref{MT-6.0}, we have expressed the sum of traces of Frobenius for the same families of elliptic curves in terms of a ${_6}G_6[\cdots]_q$ hypergeometric function. In the following theorem, we express the sum in terms of a $_{4}G_{4}[\cdots]_q$ hypergeometric function which extends \cite[Theorem 1.7]{TM}.
\begin{theorem}\label{MT-6}
	Let $p>3$ be a prime and $q=p^r,r\geq 1$. Let $E_{c,d}:y^2=x^3+cx^2+d$ and $E_{c,-d}:y^2=x^3+cx^2-d$ be elliptic curves over $\mathbb{F}_q$ such that $c,d\in\mathbb{F}_q^\times$. Then the following statements hold:
	\begin{enumerate}
		\item If $q\equiv 1,7\pmod{12}$, then we have
			\begin{align*}
				a_q(E_{c,d})+a_q(E_{c,-d})=\varphi(-3c)\cdot{_4}G_{4}\left[\begin{array}{cccc}
					0,&\frac{1}{2}, &0,& \frac{1}{2}\vspace*{0.05cm}\\
					\frac{1}{12}, & \frac{5}{12}, &\frac{7}{12}, &\frac{11}{12}
				\end{array}|\frac{729d^2}{16c^6}\right]_q.
			\end{align*}
		\item If $q\equiv 5\pmod{12}$, then we have
		\begin{align*}
			a_q(E_{c,d})+a_q(E_{c,-d})&=\varphi(c)\cdot{_4}G_{4}\left[\begin{array}{cccc}
				0, &\frac{1}{2}, & 0, & \frac{1}{2}\vspace*{0.05cm}\\
				\frac{1}{12},&\frac{5}{12},&\frac{7}{12},&\frac{11}{12}
			\end{array}|\frac{729d^2}{16c^6}\right]_q.
		\end{align*}
		\item If $p\equiv 11\pmod{12}$ and $r=1$, then we have
		\begin{align*}
				a_q(E_{c,d})+a_q(E_{c,-d})=\varphi(c)\cdot{_4}G_4\left[\begin{array}{cccc}
				0,& \frac{1}{2}, &0, &\frac{1}{2}\vspace*{0.05cm}\\
				\frac{1}{12},&\frac{5}{12},&\frac{7}{12},&\frac{11}{12}
			\end{array}|\frac{729d^2}{16c^6}\right]_p.
		\end{align*}
	\end{enumerate}
\end{theorem}
\subsection{The trace of Frobenius of elliptic curves defined over $\mathbb{Q}$}
Let $E$ be an elliptic curve defined over $\mathbb{Q}$. If $p$ is a prime for which $E$ has a good reduction $\tilde{E}$, then for each $q=p^r$ where $r\geq1$, the $q$-th trace of the Frobenius endomorphism on $E$ is defined by
\begin{align*}
	a_q(E):=q+1-\#\tilde{E}(\mathbb{F}_q),
	\end{align*}
where $\#\tilde{E}(\mathbb{F}_q)$ is the number of points on $\tilde{E}$ over $\mathbb{F}_q$ including the point at infinity. In Theorems \ref{MT-3}, \ref{MT-4}, \ref{MT-5}, \ref{MT-6.0}, and \ref{MT-6}, 
we expressed the sum of traces of Frobenius of certain families of elliptic curves as special values of ${_4}G_4[\cdots]_q$ and ${_6}G_6[\cdots]_q$ hypergeometric functions. For certain elliptic curves defined over $\mathbb{Q}$, 
it is possible to find the value of one of the traces of Frobenius appearing in the sums explicitly. This allows us to write the other trace of Frobenius appearing in the sums as  special values of ${_4}G_4[\cdots]_q$ and ${_6}G_6[\cdots]_q$ hypergeometric functions as stated in the following theorems. 
\par 
Before we state our next results, we recall a definition. For a nonzero integer $n$, we can write $n=p^mk$, where $\gcd(p,k)=1$. We then define $\text{ord}_p(n)=m$ and $\text{ord}_p(0)=\infty$. If $a=\frac{x}{y}\in\mathbb{Q}$, then we define $\text{ord}_p(a)= \text{ord}_p(x)-\text{ord}_p(y)$. 
\par The following theorem gives a $p$-adic analogue of \cite[Theorem 1.8]{TM}. We note that \cite[Theorem 1.8]{TM} is holds for all $p^r$ with $r$ even and $p\geq 5$ satisfying $p\equiv 3\pmod{4}$. Our result holds for odd values of $r$ as well. 
\begin{theorem}\label{MT-7}
	Let $E_{-\lambda}:y^2=x(x-1)(x+\lambda)$ be an elliptic curve over $\mathbb{Q}$, where $\lambda\in\{2,\frac{1}{2}\}$. If $p\geq5$ is a prime such that $p\equiv3\pmod4$, then 
	\begin{align*}
		a_{p^r}(E_{-\lambda})=\left\{\begin{array}{ll}
			\varphi(-1)\cdot{_4}G_4\left[\begin{array}{cccc}
				0, & \frac{1}{2}, & 0, & \frac{1}{2}\vspace*{0.1cm}\\
				\frac{1}{4}, & \frac{3}{4}, & \frac{1}{4}, & \frac{3}{4}
			\end{array}|\lambda^2
			\right]_{p^r},&\hbox{if $r$ is odd;}\\
			-(-p)^{\frac{r}{2}}+\varphi(-1)\cdot{_4}G_4\left[\begin{array}{cccc}
				0, & \frac{1}{2}, & 0, & \frac{1}{2}\vspace*{0.1cm}\\
				\frac{1}{4}, & \frac{3}{4}, & \frac{1}{4}, & \frac{3}{4}
			\end{array}|\lambda^2
			\right]_{p^r},&\hbox{if $r$ is even.}
		\end{array}\right.
	\end{align*}
\end{theorem}
The following theorem gives a $p$-adic analogue of \cite[Theorem 1.10]{TM} which extends \cite[Theorem 1.10]{TM} to all the odd values of $r$.
	\begin{theorem}\label{MT-8}
Let $E_{\alpha,\frac{-\alpha^3}{24}}:y^2+\alpha xy-\frac{\alpha^3}{24}y=x^3$ be an elliptic curve over $\mathbb{Q}$, where $\alpha\neq0$. 
For primes $p\equiv 5, 11 \pmod{12}$ such that $p\neq17$ and $\text{ord}_p(\alpha)=0$, we have
\begin{align*}
	a_{p^r}(E_{\alpha,-\frac{\alpha^3}{24}})=\left\{\begin{array}{ll}
		{_4}G_4\left[\begin{array}{cccc}
			0, & \frac{1}{2}, & 0, & \frac{1}{2}\vspace*{0.1cm}\\
			\frac{1}{6}, & \frac{1}{3}, & \frac{2}{3}, & \frac{5}{6}
		\end{array}|\frac{81}{64}
		\right]_{p^r},&\hbox{if $r$ is odd;}\\
		-(-p)^{\frac{r}{2}}+{_4}G_4\left[\begin{array}{cccc}
			0, & \frac{1}{2}, & 0, & \frac{1}{2}\vspace*{0.1cm}\\
			\frac{1}{6}, & \frac{1}{3}, & \frac{2}{3}, & \frac{5}{6}
		\end{array}|\frac{81}{64}
		\right]_{p^r},&\hbox{if $r$ is even.}
	\end{array}\right.
\end{align*}
\end{theorem}
The following theorem gives a $p$-adic analogue of \cite[Theorem 1.9]{TM} which extends \cite[Theorem 1.9]{TM} to all the odd values of $r$.
\begin{theorem}\label{MT-9}
	Let $E_{\alpha,\frac{-\alpha^2}{3}}:y^2=x^3+\alpha x^2-\frac{\alpha^2}{3}x$ be an elliptic curve over $\mathbb{Q}$, where $\alpha\neq0$. For primes $p\equiv 5, 11\pmod{12}$ 
	such that $\text{ord}_p(\alpha)=0$, we have
	\begin{align*}
		a_{p^r}(E_{\alpha,\frac{-\alpha^2}{3}})=\left\{\begin{array}{ll}
			\varphi(\alpha)\cdot{_4}G_4\left[\begin{array}{cccc}
				0, & \frac{1}{2}, & 0, & \frac{1}{2}\vspace*{0.1cm}\\
				\frac{1}{8}, & \frac{3}{8}, & \frac{5}{8}, & \frac{7}{8}
			\end{array}|\frac{16}{9}
			\right]_{p^r},&\hbox{if $r$ is odd;}\\
			-(-p)^{\frac{r}{2}}+\varphi(\alpha)\cdot{_4}G_4\left[\begin{array}{cccc}
				0, & \frac{1}{2}, & 0, & \frac{1}{2}\vspace*{0.1cm}\\
				\frac{1}{8}, & \frac{3}{8}, & \frac{5}{8}, & \frac{7}{8}
			\end{array}|\frac{16}{9}
			\right]_{p^r},&\hbox{if $r$ is even.}
		\end{array}\right.
	\end{align*}
\end{theorem}
In the following theorem, we express the trace of Frobenius of certain elliptic curves as a special value of a ${_6}G_6[\cdots]_q$ hypergeometric function. 
\begin{theorem}\label{MT-10}
Let $E_{\alpha,\frac{2\alpha^3}{27}}:y^2=x^3+\alpha x^2+\frac{2\alpha^3}{27}$ be an elliptic curve over $\mathbb{Q}$ such that $\alpha\neq0$. Let $p\equiv 7, 11\pmod{12}$ 
be a prime such that $\text{ord}_p(\alpha)=0$.
\begin{enumerate}
	\item If $r$ is odd, then we have
	\begin{align*}
		a_{p^r}(E_{\alpha,\frac{2\alpha^3}{27}})=\varphi(\alpha)\cdot{_6}G_6\left[\begin{array}{ccccccc}
			0,& \frac{1}{2}, & 0,& \frac{1}{2},& \frac{1}{4},& \frac{3}{4}\vspace*{0.05cm}\\
			\frac{1}{12},& \frac{1}{4},& \frac{5}{12},& \frac{7}{12},& \frac{3}{4},& \frac{11}{12}
		\end{array}|\frac{1}{4}\right]_{p^r}-\varphi(6\alpha)-\varphi(-6\alpha).
	\end{align*}
\item If $r$ is even, then we have
\begin{align*}
		a_{p^r}(E_{\alpha,\frac{2\alpha^3}{27}})&= \varphi(\alpha)\cdot{_6}G_6\left[\begin{array}{ccccccc}
			0,& \frac{1}{2}, & 0,& \frac{1}{2},& \frac{1}{4},& \frac{3}{4}\vspace*{0.05cm}\\
			\frac{1}{12},& \frac{1}{4},& \frac{5}{12},& \frac{7}{12},& \frac{3}{4},& \frac{11}{12}
		\end{array}|\frac{1}{4}\right]_{p^r}-\varphi(6\alpha)-\varphi(-6\alpha)\\
	&\hspace*{1cm}-(-p)^{\frac{r}{2}}.
\end{align*}
\end{enumerate}
\end{theorem}
It is evident that the $p$-adic hypergeometric functions appearing in Theorems \ref{MT-7}, \ref{MT-8}, \ref{MT-9} and \ref{MT-10} have integer values. Using our main results, 
one can find special values of some of the $p$-adic hypergeometric functions. We state some of the values in the following corollary though many such values can be obtained.
\begin{cor}\label{cor1}
	We have:
	\begin{enumerate}
		\item \begin{align*}
			{_4}G_4\left[\begin{array}{cccc}
				0, & \frac{1}{2}, & 0, & \frac{1}{2}\vspace*{0.1cm}\\
				\frac{1}{4}, & \frac{3}{4}, & \frac{1}{4}, & \frac{3}{4}
			\end{array}|4
			\right]_{1331}=-24\varphi(-1).
		\end{align*} 		
		\item \begin{align*}
			{_4}G_4\left[\begin{array}{cccc}
				0, & \frac{1}{2}, & 0, & \frac{1}{2}\vspace*{0.1cm}\\
				\frac{1}{6}, & \frac{1}{3}, & \frac{2}{3}, & \frac{5}{6}
			\end{array}|\frac{81}{64}
			\right]_{1331}=-39.
		\end{align*}
		\item \begin{align*}
			{_4}G_4\left[\begin{array}{cccc}
				0, & \frac{1}{2}, & 0, & \frac{1}{2}\vspace*{0.1cm}\\
				\frac{1}{8}, & \frac{3}{8}, & \frac{5}{8}, & \frac{7}{8}
			\end{array}|\frac{16}{9}
			\right]_{125}=12\varphi(3).
		\end{align*}
		\item \begin{align*}
			{_6}G_6\left[\begin{array}{ccccccc}
				0,& \frac{1}{2}, & 0,& \frac{1}{2},& \frac{1}{4},& \frac{3}{4}\vspace*{0.05cm}\\
				\frac{1}{12},& \frac{1}{4},& \frac{5}{12},& \frac{7}{12},& \frac{3}{4},& \frac{11}{12}
			\end{array}|\frac{1}{4}\right]_{1331}=-36\varphi(3)+\varphi(6)+\varphi(-6).
		\end{align*}
	\end{enumerate}
\end{cor}
\section{Preliminaries and some lemmas}\label{pre}
Let $p$ be an odd prime, and let $\mathbb{F}_q$ be the finite field containing $q$ elements, where $q=p^r,r\geq1$.  Let $\widehat{\mathbb{F}_q^{\times}}$ be the group of all the multiplicative
characters on $\mathbb{F}_q^{\times}$. We extend the domain of each $\chi\in \widehat{\mathbb{F}_q^{\times}}$ to $\mathbb{F}_q$ by setting $\chi(0):=0$
including the trivial character $\varepsilon$. For multiplicative characters $A$ and $B$ on $\mathbb{F}_q$,
the binomial coefficient ${A \choose B}$ is defined by
\begin{align*}
	{A \choose B}:=\frac{B(-1)}{q}J(A,\overline{B})=\frac{B(-1)}{q}\sum_{x \in \mathbb{F}_q}A(x)\overline{B}(1-x),
\end{align*}
where $J(A, B)$ denotes the Jacobi sum and $\overline{B}$ is the character inverse of $B$. 
 Let $\delta$ denote the function on $\widehat{\mathbb{F}_q^{\times}}$ defined by
\begin{align*}
	\delta(A):=\left\{
	\begin{array}{ll}
		1, & \hbox{if $A=\varepsilon$;} \\
		0, & \hbox{otherwise.}
	\end{array}
	\right.
\end{align*}
We recall the following properties of the binomial coefficients from \cite{greene}:
\begin{align}\label{eq-0.4}
	{A\choose \varepsilon}={A\choose A}=\frac{-1}{q}+\frac{q-1}{q}\delta(A).
\end{align}
\par
Let $\mathbb{Z}_p$ and $\mathbb{Q}_p$ denote the ring of $p$-adic integers and the field of $p$-adic numbers, respectively.
Let $\overline{\mathbb{Q}_p}$ be the algebraic closure of $\mathbb{Q}_p$ and $\mathbb{C}_p$ be the completion of $\overline{\mathbb{Q}_p}$.
Let $\mathbb{Z}_q$ be the ring of integers in the unique unramified extension of $\mathbb{Q}_p$ with residue field $\mathbb{F}_q$.
We know that $\chi\in \widehat{\mathbb{F}_q^{\times}}$ takes values in $\mu_{q-1}$, where $\mu_{q-1}$ is the group of all the $(q-1)$-th roots of unity in $\mathbb{C}^{\times}$. Since $\mathbb{Z}_q^{\times}$ contains all the $(q-1)$-th roots of unity,
we can consider multiplicative characters on $\mathbb{F}_q^\times$
to be maps $\chi: \mathbb{F}_q^{\times} \rightarrow \mathbb{Z}_q^{\times}$.
Let $\omega: \mathbb{F}_q^\times \rightarrow \mathbb{Z}_q^{\times}$ be the Teichm\"{u}ller character.
For $a\in\mathbb{F}_q^\times$, the value $\omega(a)$ is just the $(q-1)$-th root of unity in $\mathbb{Z}_q$ such that $\omega(a)\equiv a \pmod{p}$.
\par Next, we introduce the Gauss sum and recall some results. For further details, see \cite{evans}. Let $\zeta_p$ be a fixed primitive $p$-th root of unity
in $\overline{\mathbb{Q}_p}$. The trace map $\text{tr}: \mathbb{F}_q \rightarrow \mathbb{F}_p$ is given by
\begin{align}
	\text{tr}(\alpha)=\alpha + \alpha^p + \alpha^{p^2}+ \cdots + \alpha^{p^{r-1}}.\notag
\end{align}
Then the additive character
$\theta: \mathbb{F}_q \rightarrow \mathbb{Q}_p(\zeta_p)$ is defined by
\begin{align}
	\theta(\alpha)=\zeta_p^{\text{tr}(\alpha)}.\notag
\end{align}
For $\chi \in \widehat{\mathbb{F}_q^\times}$, the \emph{Gauss sum} is defined by
\begin{align}
	g(\chi):=\sum\limits_{x\in \mathbb{F}_q}\chi(x)\theta(x) .\notag
\end{align}
\begin{lemma}\emph{(\cite[(1.12)]{greene}).}\label{lemma2_1}
	For $\chi \in \widehat{\mathbb{F}_q^\times}$, we have
	$$g(\chi)g(\overline{\chi})=q\cdot \chi(-1)-(q-1)\delta(\chi).$$
\end{lemma}
\begin{theorem}\emph{(\cite[Davenport-Hasse Relation]{evans}).}\label{thm2_2}
	Let $m$ be a positive integer and let $q=p^r$ be a prime power such that $q\equiv 1 \pmod{m}$. For multiplicative characters
	$\chi, \psi \in \widehat{\mathbb{F}_q^\times}$, we have
	\begin{align}
		\prod\limits_{\chi^m=\varepsilon}g(\chi \psi)=-g(\psi^m)\psi(m^{-m})\prod\limits_{\chi^m=\varepsilon}g(\chi).\notag
	\end{align}
\end{theorem}
In \cite{greene, greene2}, Greene introduced the notion of hypergeometric functions over finite fields. He defined hypergeometric functions over finite fields using binomial coefficients as follows.
\begin{definition}(\emph{\cite[Definition 3.10]{greene}}).
	Let $n$ be a positive integer and $x\in \mathbb{F}_{q}$. For multiplicative characters $A_{0},A_{1},\dots,A_{n},B_{1}, \dots,B_{n}$ on $\mathbb{F}_{q}$, the ${_{n+1}}F_n$-hypergeometric function over $\mathbb{F}_{q}$ is defined by
	\begin{align*}
		&_{n+1}F_n\left(\begin{array}{cccc}
			A_0, & A_1, &  \ldots, & A_n \\
			& B_1, & \ldots, & B_n
		\end{array}|x
		\right)_q:=\frac{q}{q-1}\sum_{\chi\in\widehat{\mathbb{F}_{q}^{\times}}} {A_0 \chi \choose \chi}{A_{1}\chi \choose B_1\chi}\cdots {A_{n}\chi \choose B_n \chi}\chi(x).
	\end{align*}
\end{definition}In \cite{mccarthy3}, McCarthy gave another definition of hypergeometric function over finite fields using the Gauss sums.
\begin{definition}(\emph{\cite[Definition 1.4]{mccarthy3}}).
Let $n$ be a positive integer. For $x\in \mathbb{F}_{q}$ and multiplicative characters $A_{0},A_{1},\dots,A_{n},B_{1}, \dots,B_{n}$ on $\mathbb{F}_{q}$, the ${{_{n+1}}F_n }^*$-hypergeometric function over $\mathbb{F}_{q}$ is defined by
\begin{align*}
	&_{n+1}F_n\left(\begin{array}{cccc}
		A_0, & A_1, &  \ldots, & A_n \\
		& B_1, & \ldots, & B_n
	\end{array}|x
	\right)_q^\ast:=\frac{-1}{q-1}\sum_{\chi\in\widehat{\mathbb{F}_{q}^{\times}}} \prod_{i=0}^{n}\frac{g(A_i\chi)}{g(A_i)}\\
	&\hspace{6cm} \times \prod_{j=1}^n \frac{g(\overline{B_j\chi})}{g(\overline{B_j})}g(\overline{\chi})\chi(-1)^{n+1}\chi(x).
\end{align*}
\end{definition}
Now, we recall the $p$-adic gamma function. For further details, see \cite{kob}.
For a positive integer $n$,
the $p$-adic gamma function $\Gamma_p(n)$ is defined as
\begin{align}
	\Gamma_p(n):=(-1)^n\prod\limits_{0<j<n,p\nmid j}j\notag
\end{align}
and one extends it to all $x\in\mathbb{Z}_p$ by setting $\Gamma_p(0):=1$ and
\begin{align}
	\Gamma_p(x):=\lim_{x_n\rightarrow x}\Gamma_p(x_n)\notag
\end{align}
for $x\neq0$, where $x_n$ runs through any sequence of positive integers $p$-adically approaching $x$.
This limit exists, is independent of how $x_n$ approaches $x$,
and determines a continuous function on $\mathbb{Z}_p$ with values in $\mathbb{Z}_p^{\times}$.
Let $\pi \in \mathbb{C}_p$ be the fixed root of $x^{p-1} + p=0$ which satisfies
$\pi \equiv \zeta_p-1 \pmod{(\zeta_p-1)^2}$. For $x \in \mathbb{Q}$, we let $\lfloor x\rfloor$ denote the greatest integer less than or equal to $x$ and $\langle x\rangle$ 
denote the fractional part of $x$, i.e., $x-\lfloor x\rfloor$, satisfying $0\leq\langle x\rangle<1$. Then the Gross-Koblitz formula relates Gauss sums and the $p$-adic gamma function as follows.
\begin{theorem}\emph{(\cite[Gross-Koblitz]{gross}).}\label{thm2_3} For $a\in \mathbb{Z}$ and $q=p^r, r\geq 1$, we have
	\begin{align}
		g(\overline{\omega}^a)=-\pi^{(p-1)\sum\limits_{i=0}^{r-1}\langle\frac{ap^i}{q-1} \rangle}\prod\limits_{i=0}^{r-1}\Gamma_p\left(\left\langle \frac{ap^i}{q-1} \right\rangle\right).\notag
	\end{align}
\end{theorem}
We now recall a property of $\Gamma_p$ in the following proposition.
\begin{proposition}\label{prop--1}
	Let $x\in\mathbb{Z}_p$. We have $\Gamma_p(1-x)\Gamma_p(x)=(-1)^{a_0(x)}$, where $a_0(x)\in\{1,2,\ldots,p\}$ such that $x\equiv a_0(x)\pmod p$.
\end{proposition}
Next, we recall McCarthy's $p$-adic hypergeometric function. McCarthy's $p$-adic hypergeometric function $_{n}G_{n}[\cdots]_q$ is defined as follows.
\begin{definition}(\emph{\cite[Definition 5.1]{mccarthy2}}). \label{defin1}
	Let $p$ be an odd prime and $q=p^r$, $r\geq 1$. Let $t \in \mathbb{F}_q$.
	For positive integers $n$ and $1\leq k\leq n$, let $a_k$, $b_k$ $\in \mathbb{Q}\cap \mathbb{Z}_p$.
	Then the function $_{n}G_{n}[\cdots]_q$ is defined by
	\begin{align}
		&_nG_n\left[\begin{array}{cccc}
			a_1, & a_2, & \ldots, & a_n \\
			b_1, & b_2, & \ldots, & b_n
		\end{array}|t
		\right]_q\notag\\
		&\hspace{1cm}:=\frac{-1}{q-1}\sum_{a=0}^{q-2}(-1)^{an}~~\overline{\omega}^a(t)
		\prod\limits_{k=1}^n\prod\limits_{i=0}^{r-1}(-p)^{-\lfloor \langle a_kp^i \rangle-\frac{ap^i}{q-1} \rfloor -\lfloor\langle -b_kp^i \rangle +\frac{ap^i}{q-1}\rfloor}\notag\\
	&\hspace{2cm} \times \frac{\Gamma_p(\langle (a_k-\frac{a}{q-1})p^i\rangle)}{\Gamma_p(\langle a_kp^i \rangle)}
		\frac{\Gamma_p(\langle (-b_k+\frac{a}{q-1})p^i \rangle)}{\Gamma_p(\langle -b_kp^i \rangle)}.\notag
	\end{align}
\end{definition}
We state a product formula for the $p$-adic gamma function which follows from \cite[Theorem 3.1]{gross}. If $m\in\mathbb{Z}^+$, $p\nmid m$ and $x(q-1)\in\mathbb{Z}$, then we have
\begin{align}\label{eq-3}
	\prod_{i=0}^{r-1}\prod_{h=0}^{m-1}\Gamma_p\left(\left\langle\left(\frac{x+h}{m}\right)p^i\right\rangle\right)=\omega(m^{(1-x)(1-q)})\prod_{i=0}^{r-1}\Gamma_p\left(\left\langle xp^i\right\rangle\right)\prod_{h=1}^{m-1}\Gamma_p\left(\left\langle\frac{hp^i}{m}\right\rangle\right).
\end{align}
We recall some lemmas which relate certain products of values of the $p$-adic gamma function. We will use these lemmas in the proof of our main results.
\begin{lemma}\emph{(\cite[Lemma 3.1]{BS1}).}\label{lemma-3_1}
	Let $p$ be a prime and $q=p^r, r\geq 1$. For $0\leq a\leq q-2$ and $t\geq 1$ with $p\nmid t$, we have
	\begin{align}
		\omega(t^{-ta})\prod\limits_{i=0}^{r-1}\Gamma_p\left(\left\langle\frac{-tp^ia}{q-1}\right\rangle\right)
		\prod\limits_{h=1}^{t-1}\Gamma_p\left(\left\langle \frac{hp^i}{t}\right\rangle\right)
		=\prod\limits_{i=0}^{r-1}\prod\limits_{h=0}^{t-1}\Gamma_p\left(\left\langle\frac{p^i(1+h)}{t}-\frac{p^ia}{q-1}\right\rangle \right).\notag
	\end{align}
\end{lemma}
\begin{lemma}\emph{(\cite[Lemma 3.2]{BS1}).}\label{lemma-3_2}
	Let $p$ be a prime and $q=p^r, r\geq 1$. For $0\leq a\leq q-2$ and $t\geq 1$ with $p\nmid t$, we have
	\begin{align*}
		\omega(t^{ta})\prod\limits_{i=0}^{r-1}\Gamma_p\left(\left\langle\frac{tp^ia}{q-1}\right\rangle\right)
		\prod\limits_{h=1}^{t-1}\Gamma_p\left(\left\langle \frac{hp^i}{t}\right\rangle\right)
		=\prod\limits_{i=0}^{r-1}\prod\limits_{h=0}^{t-1}\Gamma_p\left(\left\langle\frac{p^i h}{t}+\frac{p^ia}{q-1}\right\rangle \right).\notag
	\end{align*}
\end{lemma}
\begin{lemma}\emph{(\cite[Lemma 3.4]{BSM})}.\label{lemma-3_5}
Let $p$ be an odd prime and $q=p^{r}, r\geq 1$. For $0<a\leq q-2$, we have
\begin{align}\label{eq-12}
\prod_{i=0}^{r-1} \Gamma_{p}\left(\left\langle\left(1-\frac{a}{q-1}\right)p^{i}\right\rangle\right)\Gamma_{p}\left(\left\langle\frac{ap^{i}}{q-1}\right\rangle\right) = (-1)^r \overline{\omega}^{a}(-1).
\end{align}
For $0\leq a\leq q-2$ such that $a\neq \frac{q-1}{2}$, we have
\begin{align}\label{eq-13}
\prod_{i=0}^{r-1} \frac{\Gamma_{p}(\langle(\frac{1}{2}-\frac{a}{q-1})p^{i}\rangle)\Gamma_{p}(\langle(\frac{1}{2}+\frac{a}{q-1})p^{i}\rangle)}{\Gamma_{p}(\langle\frac{p^{i}}{2}\rangle)\Gamma_{p}(\langle\frac{p^{i}}{2}\rangle)} = \overline{\omega}^{a}(-1).
\end{align}
\end{lemma}
Finally, we recall two lemmas relating fractional and integral parts of certain rational numbers which will be used to simplify certain products of the $p$-adic gamma function.
\begin{lemma}\emph{(\cite[Lemma 2.6]{SB}).}\label{lemma-3_3}
	Let $p$ be an odd prime and $q=p^r, r\geq 1$. Let $d\geq2$ be an integer such that $p\nmid d$. Then, for $1\leq a\leq q-2$ and $0\leq i\leq r-1$, we have
	\begin{align*}
	\left\lfloor\frac{ap^i}{q-1}\right\rfloor +\left\lfloor\frac{-dap^i}{q-1}\right\rfloor = \sum_{h=1}^{d-1} \left\lfloor\left\langle\frac{hp^i}{d}\right\rangle-\frac{ap^i}{q-1}\right\rfloor -1.
	\end{align*}
\end{lemma}
\begin{lemma}\emph{(\cite[Lemma 2.7]{SB}).}\label{lemma-3_4}
	Let $p$ be an odd prime and $q=p^r, r\geq 1$. Let $l$ be a positive integer such that $p\nmid l$. Then, for $0\leq a\leq q-2$ and $0\leq i\leq r-1$, we have
	\begin{align*}
	\left\lfloor\frac{lap^i}{q-1}\right\rfloor = \sum_{h=0}^{l-1} \left\lfloor\left\langle\frac{-hp^i}{l}\right\rangle+\frac{ap^i}{q-1}\right\rfloor.
	\end{align*}
\end{lemma}
We prove the following lemma relating fractional and integral parts of certain rational numbers.
\begin{lemma}\label{lemma-12}
Let	$p$ be an odd prime and $q=p^r$, $r\geq1$. Let $x=\frac{m}{d}$ be a rational number such that $\gcd(m,d)=1$. For $0\leq j\leq q-2$ and $0\leq i\leq r-1$, we have
\begin{align}
	\label{eqn-new-02}\left\lfloor\left\langle xp^i\right\rangle-\frac{2jp^i}{q-1}\right\rfloor=\left\lfloor\left\langle \frac{xp^i}{2}\right\rangle-\frac{jp^i}{q-1}\right\rfloor+\left\lfloor\left\langle \frac{(1+x)p^i}{2}\right\rangle-\frac{jp^i}{q-1}\right\rfloor,\\
	\label{eqn-new-03}\left\lfloor\left\langle xp^i\right\rangle+\frac{2jp^i}{q-1}\right\rfloor=\left\lfloor\left\langle \frac{xp^i}{2}\right\rangle+\frac{jp^i}{q-1}\right\rfloor+\left\lfloor\left\langle \frac{(1+x)p^i}{2}\right\rangle+\frac{jp^i}{q-1}\right\rfloor.
\end{align} 
	\end{lemma}
\begin{proof}
We give a proof of \eqref{eqn-new-02}.	Let $y:=\left\langle xp^i\right\rangle$. Then $\frac{y}{2}$ is equal to either $\left\langle \frac{xp^i}{2}\right\rangle$ or $\left\langle \frac{xp^i}{2}\right\rangle-\frac{1}{2}$. If $y=\left\langle xp^i\right\rangle$ and  $\frac{y}{2}=\left\langle \frac{xp^i}{2}\right\rangle$, 
	then we have $\left\langle\frac{(1+x)p^i}{2}\right\rangle=\frac{y+1}{2}$. If $y=\left\langle xp^i\right\rangle$ and  $\frac{y}{2}=\left\langle \frac{xp^i}{2}\right\rangle-\frac{1}{2}$, then we have $\left\langle\frac{(1+x)p^i}{2}\right\rangle=\frac{y}{2}$.
	Therefore, to obtain the required identity we need to prove that 
	\begin{align}\label{eqn-0}
		\left\lfloor y-\frac{2jp^i}{q-1}\right\rfloor=	\left\lfloor \frac{y}{2}-\frac{jp^i}{q-1}\right\rfloor+	\left\lfloor\frac{1+y}{2}-\frac{jp^i}{q-1}\right\rfloor.
	\end{align}
We can write $\left\lfloor y-\frac{2jp^i}{q-1}\right\rfloor=2k+s$, where $k\in\mathbb{Z}$ and $s\in\{0,1\}$. This yields
\begin{align*}
	2k+s\leq y-\frac{2jp^i}{q-1}< 2k+s+1,
	\end{align*}
and hence 
\begin{align}
	&k+\frac{s}{2}\leq \frac{y}{2}-\frac{jp^i}{q-1}< k+\frac{s+1}{2}\label{eqn-1},\\
		&k+\frac{s+1}{2}\leq \frac{y+1}{2}-\frac{jp^i}{q-1}< k+1+\frac{s}{2}\label{eqn-2}.
\end{align}
From \eqref{eqn-1} and \eqref{eqn-2}, we have $\left\lfloor\frac{y}{2}-\frac{jp^i}{q-1}\right\rfloor=k$ and $\left\lfloor\frac{y+1}{2}-\frac{jp^i}{q-1}\right\rfloor=k+s$, respectively. Therefore, both the sides of \eqref{eqn-0} are equal to $2k+s$. This completes the proof of \eqref{eqn-new-02}. 
The proof of \eqref{eqn-new-03} follows similarly.
	\end{proof}
In the following lemmas, we evaluate two expressions containing certain values of the $p$-adic gamma function. 
\begin{lemma}\label{lemma-13}
	Let $p$ be an odd prime and $q=p^{r}, r\geq 1$ such that $q\equiv 1\pmod4$. For $0\leq n\leq q-2$ such that $n\notin \{\frac{3(q-1)}{4}, \frac{q-1}{4}\}$, we have
		\begin{align*}
		(-p)^{\sum_{i=0}^{r-1}s_{i,n}}\prod_{i=0}^{r-1}\frac{ \Gamma_{p}\left(\left\langle\left(\frac{1}{4}+\frac{n}{q-1}\right)p^{i}\right\rangle\right)\Gamma_{p}\left(\left\langle\left(\frac{3}{4}-\frac{n}{q-1}\right)p^{i}\right\rangle\right)}{\Gamma_{p}
		\left(\left\langle\frac{3p^i}{4}\right\rangle\right)\Gamma_{p}\left(\left\langle\frac{p^i}{4}\right\rangle\right)}\notag\\
		\times\frac{\Gamma_{p}\left(\left\langle\left(\frac{1}{4}-\frac{n}{q-1}\right)p^{i}\right\rangle\right)\Gamma_{p}\left(\left\langle\left(\frac{3}{4}+\frac{n}{q-1}\right)p^{i}\right\rangle\right)}{\Gamma_{p}\left(\left\langle\frac{3p^i}{4}\right\rangle\right)
		\Gamma_{p}\left(\left\langle\frac{p^i}{4}\right\rangle\right)} = 1,
	\end{align*}
	where
	$s_{i,n}=-\left\lfloor\frac{3}{4}-\frac{np^i}{q-1}\right\rfloor-\left\lfloor\frac{1}{4}+\frac{np^i}{q-1}\right\rfloor-\left\lfloor\frac{3}{4}+\frac{np^i}{q-1}\right\rfloor-\left\lfloor\frac{1}{4}-\frac{np^i}{q-1}\right\rfloor$.
	\end{lemma}
\begin{proof}
	Since $q\equiv1\pmod4$, there exists a character $\chi_4$ of order $4$, and the inverse of $\chi_4$ is $\overline{\chi_4}=\chi_4^3$. Let $\psi\in\widehat{\mathbb{F}_q^\times}$ be such that $\psi\neq\chi_4,\overline{\chi_4}$. Then, using Lemma \ref{lemma2_1}, we have
	\begin{align*}
		\frac{g(\chi_4\psi)g(\overline{\chi_4\psi})g(\overline{\chi_4}\psi)g(\chi_4\overline{\psi})}{g(\chi_4)g(\overline{\chi_4})g(\chi_4)g(\overline{\chi_4})}=\frac{q^2\overline{\chi_4}\psi\chi_4\psi(-1)}{q^2\overline{\chi_4}\chi_4(-1)}=1.
		\end{align*}
	Taking $\chi_4=\overline{\omega}^{\frac{q-1}{4}}$ and $\psi =\overline{\omega}^{n}$ for some $n$ such that $0\leq n\leq q-2$ and $n\neq\frac{3(q-1)}{4},\frac{q-1}{4}$, and then employing Gross-Koblitz formula, we obtain the desired result.
	\end{proof}
\begin{lemma}\label{lemma-14}
	Let $p\equiv-1\pmod d$. Then for $0\leq n\leq p-2$, we have
	\begin{align*}
		\frac{\Gamma_p\left(\left\langle\frac{-1}{d}+\frac{n}{p-1}\right\rangle\right)\Gamma_p\left(\left\langle\frac{-(d-1)}{d}+\frac{n}{p-1}\right\rangle\right)\Gamma_p\left(\left\langle\frac{1}{d}-\frac{n}{p-1}\right\rangle\right)
		\Gamma_p\left(\left\langle\frac{d-1}{d}-\frac{n}{p-1}\right\rangle\right)}{\Gamma_p\left(\left\langle\frac{1}{d}\right\rangle\right)^2\Gamma_p\left(\left\langle\frac{d-1}{d}\right\rangle\right)^2}=1.
	\end{align*}
\end{lemma}
\begin{proof}
	Suppose that $d=2$. Then for $n\neq\frac{p-1}{2}$, using \eqref{eq-13} with $r=1$, we obtain the required result. If $n=\frac{p-1}{2}$, then using Proposition \ref{prop--1}, we obtain the required result.
	Now, we consider $d\neq2$. Let $x:=\frac{1}{d}+\frac{n}{p-1}$ and $y:=\frac{d-1}{d}+\frac{n}{p-1}$. Then $1-x=\frac{d-1}{d}-\frac{n}{p-1}$ and $1-y=\frac{1}{d}-\frac{n}{p-1}$. 
	Using Proposition \ref{prop--1}, we obtain $\Gamma_p\left(\left\langle\frac{1}{d}\right\rangle\right)^2\Gamma_p\left(\left\langle\frac{d-1}{d}\right\rangle\right)^2=1$. 
	Note that for $d\neq2$, $x$ and $y$ are not integers and hence $\langle1-x\rangle=1-\langle x\rangle$ and $\langle1-y\rangle=1-\langle y\rangle$. Thus, we have
	\begin{align*}
		\Gamma_p\left(\left\langle x\right\rangle\right)\Gamma_p\left(\left\langle 1-x\right\rangle\right)\Gamma_p\left(\left\langle y\right\rangle\right)\Gamma_p\left(\left\langle 1-y\right\rangle\right)&=
		\Gamma_p\left(\left\langle x\right\rangle\right)\Gamma_p\left(1-\left\langle x\right\rangle\right)\Gamma_p\left(\left\langle y\right\rangle\right)\Gamma_p\left(1-\left\langle y\right\rangle\right)\\
		&=(-1)^{a_0(\left\langle x\right\rangle)+a_0(\left\langle y\right\rangle)}.
	\end{align*}
	Hence, to complete the proof of the lemma, we need to prove that $$(-1)^{a_0(\left\langle x\right\rangle)+a_0(\left\langle y\right\rangle)}=1.$$
	\underline{Case 1:} $0\leq n<\frac{p-1}{d}$.\\
Clearly, $\left\langle x\right\rangle=\frac{1}{d}+\frac{n}{p-1}$ and $\left\langle y\right\rangle=\frac{d-1}{d}+\frac{n}{p-1}$. Also, we have
	\begin{align*}
		\frac{1}{d}+\frac{n}{p-1}-\left(\frac{p+1}{d}-n\right)\equiv0\pmod p, \\
		\frac{d-1}{d}+\frac{n}{p-1}-\left(\frac{(d-1)(p+1)}{d}-n\right)\equiv0\pmod p.
	\end{align*}
	Furthermore, $\frac{p+1}{d}-n,\frac{(d-1)(p+1)}{d}-n\in\{1,2,\ldots,p\}$. Thus, $a_0(\left\langle x\right\rangle)=\frac{p+1}{d}-n$ and $a_{0}(\left\langle y\right\rangle)=\frac{(d-1)(p+1)}{d}-n$. 
	Therefore, $a_0(\left\langle x\right\rangle)+a_0(\left\langle y\right\rangle)$ is an even number and hence, we are done for this case.\\
	\underline{Case 2:} $\frac{p-1}{d}< n<\frac{(d-1)(p-1)}{d}$.\\
	 Clearly, $\left\langle x\right\rangle=\frac{1}{d}+\frac{n}{p-1}$ and $\left\langle y\right\rangle=-\frac{1}{d}+\frac{n}{p-1}$. Also, we have
	\begin{align*}
		\frac{1}{d}+\frac{n}{p-1}-\left(\frac{(d+1)p+1}{d}-n\right)\equiv0\pmod p, \\
		-\frac{1}{d}+\frac{n}{p-1}-\left(\frac{(d-1)p-1}{d}-n\right)\equiv0\pmod p.
	\end{align*} 
	Furthermore, $\frac{(d+1)p+1}{d}-n,\frac{(d-1)p-1}{d}-n\in\{1,2,\ldots,p\}$. Thus, $a_0(\left\langle x\right\rangle)=\frac{(d+1)p+1}{d}-n$ and $a_{0}(\left\langle y\right\rangle)=\frac{(d-1)p-1}{d}-n$. 
	Therefore, $a_0(\left\langle x\right\rangle)+a_0(\left\langle y\right\rangle)$ is an even number and hence, we are done for this case.\\
	\underline{Case 3:} $\frac{(d-1)(p-1)}{d}< n< p-1$. \\
	Clearly, $\left\langle x\right\rangle=-\frac{d-1}{d}+\frac{n}{p-1}$ and $\left\langle y\right\rangle=-\frac{1}{d}+\frac{n}{p-1}$. Also, we have
	\begin{align*}
		-\frac{d-1}{d}+\frac{n}{p-1}-\left(\frac{(d+1)p-(d-1)}{d}-n\right)\equiv0\pmod p, \\
		-\frac{1}{d}+\frac{n}{p-1}-\left(\frac{(2d-1)p-1}{d}-n\right)\equiv0\pmod p.
	\end{align*} 
	Furthermore, $\frac{(d+1)p-(d-1)}{d}-n,\frac{(2d-1)p-1}{d}-n\in\{1,2,\ldots,p\}$. Thus, $a_0(\left\langle x\right\rangle)=\frac{(d+1)p-(d-1)}{d}-n$ and $a_{0}(\left\langle y\right\rangle)=\frac{(2d-1)p-1}{d}-n$. 
	Therefore, $a_0(\left\langle x\right\rangle)+a_0(\left\langle y\right\rangle)$ is an even number. This completes the proof of the lemma.
\end{proof}
\section{Proof of Theorem \ref{MT-1} and Theorem \ref{thrm-1}}
In this section we prove both the identities for the $p$-adic hypergeometric functions. These identities will be used to prove some of our main results.
\begin{proof}[Proof of Theorem \ref{MT-1}]
For $x\in \mathbb{F}_q$, let 
\begin{align*}
	A_x:&={_2}G_2\left[\begin{array}{cc}
		a_1,\hspace*{-0.15cm} & a_2 \\
		a_3,\hspace*{-0.15cm} & a_4
	\end{array}|x \right]_q + {_2}G_2\left[\begin{array}{cc}
		a_1,\hspace*{-0.15cm} & a_2 \\
		a_3,\hspace*{-0.15cm} & a_4
	\end{array}|-x
	\right]_q.
	\end{align*} 
	Then,
	\begin{align*}
	A_x&=-\frac{1}{q-1}\sum\limits_{j=0}^{q-2}(\overline{\omega}^j(x)+\overline{\omega}^j(-x))(-p)^{-\sum_{i=0}^{r-1}\left(\sum_{k=1}^{2}\left\lfloor \langle a_kp^i\rangle-\frac{jp^i}{q-1}\right\rfloor+\sum_{k=3}^{4}\left\lfloor \langle -a_kp^i\rangle+\frac{jp^i}{q-1}\right\rfloor\right)}\\
	&\hspace*{.5cm}\times\prod_{i=0}^{r-1}\frac{\prod_{k=1}^{2}\Gamma_p\left(\left\langle a_kp^i-\frac{jp^i}{q-1}\right\rangle\right)\prod_{k=3}^{4}\Gamma_p\left(\left\langle -a_kp^i+\frac{jp^i}{q-1}\right\rangle\right)}{\Gamma_p(\left\langle a_1p^i\right\rangle)
	\Gamma_p(\left\langle a_2p^i\right\rangle)\Gamma_p(\left\langle -a_3p^i\right\rangle)\Gamma_p(\left\langle -a_4p^i\right\rangle)}.
\end{align*}
Since $\overline{\omega}^j(-1)=-1$ if $j$ is odd and $\overline{\omega}^j(-1)=1$ if $j$ is even, we have
\begin{align}\label{eq-4}
	A_x&=-\frac{2}{q-1}\sum\limits_{j=0}^{\frac{q-3}{2}}\overline{\omega}^{2j}(x)(-p)^{-\sum_{i=0}^{r-1}\left(\sum_{k=1}^{2}\left\lfloor \langle a_kp^i\rangle-\frac{2jp^i}{q-1}\right\rfloor+\sum_{k=3}^{4}\left\lfloor \langle -a_kp^i\rangle+\frac{2jp^i}{q-1}\right\rfloor\right)}\notag\\
	&\hspace*{.5cm}\times\prod_{i=0}^{r-1}\frac{\prod_{k=1}^{2}\Gamma_p\left(\left\langle a_kp^i-\frac{2jp^i}{q-1}\right\rangle\right)\prod_{k=3}^{4}\Gamma_p\left(\left\langle -a_kp^i+\frac{2jp^i}{q-1}\right\rangle\right)}{\Gamma_p(\left\langle a_1p^i\right\rangle)
	\Gamma_p(\left\langle a_2p^i\right\rangle)\Gamma_p(\left\langle -a_3p^i\right\rangle)\Gamma_p(\left\langle -a_4p^i\right\rangle)}.
\end{align}
We have $d_k|(q-1)$ for $k=1,\ldots, 4$. Using \eqref{eq-3} with $m=2$ and $x=a_k-\frac{2j}{q-1}$ for $k=1, 2$, we obtain 
\begin{align}\label{eq-5}
	\prod_{i=0}^{r-1}\Gamma_p\left(\left\langle\left(a_k-\frac{2j}{q-1}\right)p^i\right\rangle\right)&=\overline{\omega}\left(2^{(1-a_k+\frac{2j}{q-1})(1-q)}\right)\prod_{i=0}^{r-1}
	\frac{\Gamma_p\left(\left\langle\left(\frac{a_k}{2}-\frac{j}{q-1}\right)p^i\right\rangle\right)}{\Gamma_p\left(\left\langle\frac{p^i}{2}\right\rangle\right)}\notag\\
	&\hspace*{1cm}\times\Gamma_p\left(\left\langle\left(\frac{1+a_k}{2}-\frac{j}{q-1}\right)p^i\right\rangle\right).
\end{align}
 Again, using \eqref{eq-3} with $m=2$ and $x=-a_k+\frac{2j}{q-1}$ for $k=3, 4$, we obtain
\begin{align}\label{eq--5}
	\prod_{i=0}^{r-1}\Gamma_p\left(\left\langle\left(-a_k+\frac{2j}{q-1}\right)p^i\right\rangle\right)&=\overline{\omega}\left(2^{(1+a_k-\frac{2j}{q-1})(1-q)}\right)\prod_{i=0}^{r-1}
	\frac{\Gamma_p\left(\left\langle\left(\frac{-a_k}{2}+\frac{j}{q-1}\right)p^i\right\rangle\right)}{\Gamma_p\left(\left\langle\frac{p^i}{2}\right\rangle\right)}\notag\\
	&\hspace*{1cm}\times\Gamma_p\left(\left\langle\left(\frac{1-a_k}{2}+\frac{j}{q-1}\right)p^i\right\rangle\right).
\end{align}
Using \eqref{eqn-new-02} with $x=a_k$ for $k=1,2$ and \eqref{eqn-new-03} with $x=-a_k$ for $k=3,4$, and then adding the obtained equations, we deduce that
\begin{align}\label{eq-7}
	&\sum_{k=1}^{2}\left\lfloor \langle a_kp^i\rangle-\frac{2jp^i}{q-1}\right\rfloor+\sum_{k=3}^{4}\left\lfloor \langle -a_kp^i\rangle+\frac{2jp^i}{q-1}\right\rfloor \notag\\
	&=	\sum_{k=1}^{2}\left(\left\lfloor\left\langle \frac{a_kp^i}{2}\right\rangle-\frac{jp^i}{q-1}\right\rfloor+
	\left\lfloor\left\langle \frac{(1+a_k)p^i}{2}\right\rangle-\frac{jp^i}{q-1}\right\rfloor\right)\notag\\
	&+\sum_{k=3}^{4}\left(
	\left\lfloor\left\langle \frac{-a_kp^i}{2}\right\rangle+\frac{jp^i}{q-1}\right\rfloor+	\left\lfloor\left\langle \frac{(1-a_k)p^i}{2}\right\rangle+\frac{jp^i}{q-1}\right\rfloor\right).
\end{align}
Substituting the expressions \eqref{eq-5}, \eqref{eq--5}, and \eqref{eq-7} in \eqref{eq-4}, and then using the fact that $\Gamma_p\left(\left\langle\frac{p^i}{2}\right\rangle\right)^4=1$, we obtain
\begin{align*}
	A_x&=-\frac{2\alpha}{q-1}\sum\limits_{j=0}^{\frac{q-3}{2}}\overline{\omega}^{2j}(x)(-p)^{-\sum_{i=0}^{r-1}s_{i,j}}\notag\\
	&\times \prod_{i=0}^{r-1}\frac{N_{i,j}}{\Gamma_p(\left\langle a_1p^i\right\rangle)\Gamma_p(\left\langle a_2p^i\right\rangle)\Gamma_p(\left\langle -a_3p^i\right\rangle)\Gamma_p(\left\langle -a_4p^i\right\rangle)},
	\end{align*}
where $\alpha=\overline{\omega}(2^{(a_1+a_2-a_3-a_4)(q-1)})$ and 
\begin{align*}
	s_{i,j}&=	\sum_{k=1}^{2}\left(\left\lfloor\left\langle \frac{a_kp^i}{2}\right\rangle-\frac{jp^i}{q-1}\right\rfloor+\left\lfloor\left\langle \frac{(1+a_k)p^i}{2}\right\rangle-\frac{jp^i}{q-1}\right\rfloor\right)\\
	&+\sum_{k=3}^{4}\left(
\left\lfloor\left\langle \frac{-a_kp^i}{2}\right\rangle+\frac{jp^i}{q-1}\right\rfloor
+\left\lfloor\left\langle \frac{(1-a_k)p^i}{2}\right\rangle+\frac{jp^i}{q-1}\right\rfloor\right),\\
N_{i,j}&=\prod_{k=1}^{2}\Gamma_p\left(\left\langle\left(\frac{a_k}{2}-\frac{j}{q-1}\right)p^i\right\rangle\right)\Gamma_p\left(\left\langle\left(\frac{1+a_k}{2}-\frac{j}{q-1}\right)p^i\right\rangle\right)\\
&\times\prod_{k=3}^{4}\Gamma_p\left(\left\langle\left(\frac{-a_k}{2}+\frac{j}{q-1}\right)p^i\right\rangle\right)\Gamma_p\left(\left\langle\left(\frac{1-a_k}{2}+\frac{j}{q-1}\right)p^i\right\rangle\right).
\end{align*}
We can rewrite $A_x$ as
\begin{align*}
	A_x&=-\frac{2\alpha}{q-1}\sum\limits_{j=0}^{\frac{q-3}{2}}B_j\\
	&=-\frac{\alpha}{q-1}\sum\limits_{j=0}^{\frac{q-3}{2}}B_j-\frac{\alpha}{q-1}\sum\limits_{j=0}^{\frac{q-3}{2}}B_j\\
	&=-\frac{\alpha}{q-1}\sum\limits_{j=0}^{\frac{q-3}{2}}B_j-\frac{\alpha}{q-1}\sum\limits_{j=\frac{q-1}{2}}^{q-2}B_{j-\frac{q-1}{2}},
	\end{align*}
where $B_j=\overline{\omega}^{2j}(x)(-p)^{-\sum_{i=0}^{r-1}s_{i,j}}\prod_{i=0}^{r-1}\frac{N_{i,j}}{\Gamma_p(\left\langle a_1p^i\right\rangle)\Gamma_p(\left\langle a_2p^i\right\rangle)
\Gamma_p(\left\langle -a_3p^i\right\rangle)\Gamma_p(\left\langle -a_4p^i\right\rangle)}$. We can easily check that $N_{j-\frac{q-1}{2}}=N_j$. Also,
\begin{align*}
	s_{i,j-\frac{q-1}{2}}&=\sum_{k=1}^{2}\left(\left\lfloor\left\langle \frac{a_kp^i}{2}\right\rangle-\frac{jp^i}{q-1}+\frac{1}{2}\right\rfloor+\left\lfloor\left\langle \frac{(1+a_k)p^i}{2}\right\rangle-\frac{jp^i}{q-1}+\frac{1}{2}\right\rfloor\right)\\
	&+\sum_{k=3}^{4}\left(
	\left\lfloor\left\langle \frac{-a_kp^i}{2}\right\rangle+\frac{jp^i}{q-1}-\frac{1}{2}\right\rfloor
	+\left\lfloor\left\langle \frac{(1-a_k)p^i}{2}\right\rangle+\frac{jp^i}{q-1}-\frac{1}{2}\right\rfloor\right).
	\end{align*}
	Let
\begin{align*}
 y_k=\left\{
	\begin{array}{ll}
		\left\langle a_kp^i\right\rangle, & \hbox{if $k=1,2$;} \vspace{0.2cm}\\
		\left\langle -a_kp^i\right\rangle, & \hbox{if $k=3,4$.}
	\end{array}
	\right.
\end{align*}
Then, 
\begin{align*}
	s_{i,j-\frac{q-1}{2}}&=\sum_{k=1}^{2}\left(\left\lfloor \frac{y_k}{2}-\frac{jp^i}{q-1}+\frac{1}{2}\right\rfloor+\left\lfloor \frac{1+y_k}{2}-\frac{jp^i}{q-1}+\frac{1}{2}\right\rfloor\right)\\
	&+\sum_{k=3}^{4}\left(
	\left\lfloor\frac{y_k}{2}+\frac{jp^i}{q-1}-\frac{1}{2}\right\rfloor
	+\left\lfloor\frac{1+y_k}{2}+\frac{jp^i}{q-1}-\frac{1}{2}\right\rfloor\right)\\
	&=s_{i,j}.
	\end{align*}
Hence, $B_{j-\frac{q-1}{2}}=B_j$.
Therefore, we have
\begin{align}\label{eq--8}
	A_x=-\frac{\alpha}{q-1}\sum\limits_{j=0}^{q-2}B_j.
	\end{align}
Using \eqref{eq-3} with $m=2$ and $x=a_k$ for $k=1, 2$; and with $m=2$ and $x=-a_k$ for $k=3, 4$, and  then multiplying all the obtained equations, we have
\begin{align}\label{eq--9}
	&\prod_{i=0}^{r-1}\left(\prod_{k=1}^{2}\Gamma_p\left(\left\langle a_kp^i\right\rangle\right)\right)\left(\prod_{k=3}^{4}\Gamma_p\left(\left\langle -a_kp^i\right\rangle\right)\right)\notag\\
	&=\alpha\prod_{i=0}^{r-1}\prod_{k=1}^{2}\Gamma_p\left(\left\langle\frac{a_kp^i}{2}\right\rangle\right)\Gamma_p\left(\left\langle\frac{(1+a_k)p^i}{2}\right\rangle\right)\notag\\
	&\times\prod_{k=3}^{4}\Gamma_p\left(\left\langle\frac{(-a_k)p^i}{2}\right\rangle\right)\Gamma_p\left(\left\langle\frac{(1-a_k)p^i}{2}\right\rangle\right).
	\end{align}
Substituting \eqref{eq--9} in \eqref{eq--8}, we obtain
\begin{align*}
	A_x={_4}G_4\left[\begin{array}{cccc}
		\frac{a_1}{2},\hspace*{-0.15cm} & \frac{1+a_1}{2},\hspace*{-0.15cm} & 	\frac{a_2}{2},\hspace*{-0.15cm} & \frac{1+a_2}{2}\hspace*{-0.15cm}\vspace{.15cm}\\
		\frac{a_3}{2},\hspace*{-0.15cm} & \frac{1+a_3}{2},\hspace*{-0.15cm} &\hspace*{-0.15cm} 	\frac{a_4}{2},\hspace*{-0.15cm} & \frac{1+a_4}{2}\hspace*{-0.15cm}
	\end{array}|x^2
	\right]_q.
	\end{align*}
This completes the proof of the theorem.
\end{proof}
Next, we prove Theorem \ref{thrm-1}.
\begin{proof}[Proof of Theorem \ref{thrm-1}]
	We have
	\begin{align}\label{eqn-new-01}
		&{_{n+2}}G_{n+2}\left[\begin{array}{ccccc}
			a_1, &  \ldots, & a_n, &\frac{1}{d}, &\frac{d-1}{d} \\
			b_1, &  \ldots, & b_n, &\frac{1}{d}, &\frac{d-1}{d}
		\end{array}|t
		\right]_p\notag\\ &=-\frac{1}{p-1}\sum_{j=0}^{p-2}(-1)^{(n+2)j}\overline{\omega}^j(t)(-p)^{\alpha_{j}+\beta_{j}}M_jN_j,
	\end{align}
	where
	\begin{align*}
		\alpha_{j}&=-\sum_{k=1}^{n}\left(\left\lfloor \langle a_k\rangle-\frac{j}{p-1}\right\rfloor+\left\lfloor \langle-b_k\rangle+\frac{j}{p-1}\right\rfloor\right),\\
		\beta_{j}&=-\left\lfloor \frac{1}{d}-\frac{j}{p-1}\right\rfloor-\left\lfloor \left\langle -\frac{1}{d}\right\rangle+\frac{j}{p-1}\right\rfloor-\left\lfloor \frac{d-1}{d}-\frac{j}{p-1}\right\rfloor\\
		&\hspace{.3cm}-\left\lfloor\left\langle -\frac{d-1}{d}\right\rangle+\frac{j}{p-1}\right\rfloor,\\
		M_j&=\prod_{k=1}^{n}\frac{\Gamma_p\left(\left\langle a_k-\frac{j}{p-1}\right\rangle\right)\Gamma_p\left(\left\langle -b_k+\frac{j}{p-1}\right\rangle\right)}{\Gamma_p\left(\left\langle a_k\right\rangle\right)\Gamma_p\left(\left\langle -b_k\right\rangle\right)},\\
		N_j&=\frac{\Gamma_p\left(\left\langle \frac{1}{d}-\frac{j}{p-1}\right\rangle\right)\Gamma_p\left(\left\langle-\frac{1}{d}+\frac{j}{p-1}\right\rangle\right)
		\Gamma_p\left(\left\langle\frac{d-1}{d}-\frac{j}{p-1}\right\rangle\right)\Gamma_p\left(\left\langle-\frac{d-1}{d}+\frac{j}{p-1}\right\rangle\right)}{\Gamma_p\left(\left\langle \frac{1}{d}\right\rangle\right)^2\Gamma_p\left(\left\langle \frac{d-1}{d}\right\rangle\right)^2}.
	\end{align*}
	Let $a:=\left\lfloor\frac{p-1}{d}\right\rfloor$ and $b:=\left\lfloor\frac{(d-1)(p-1)}{d}\right\rfloor$. Taking $j$ in the intervals $\left[0,a\right]$, $\left[a+1,b\right]$, and $\left[b+1,p-2\right]$ respectively,
	we obtain $\beta_j=0$. Using Lemma \ref{lemma-14}, we have $N_j=1$. Substituting these values in \eqref{eqn-new-01}, we complete the proof of the theorem.
\end{proof}
\section{Proof of Theorems \ref{MT-3}, \ref{MT-4}, \ref{MT-5}, \ref{MT-6.0} and \ref{MT-6}}
Throughout this section, $T$ denotes a generator of $\widehat{\mathbb{F}_q^{\times}}$.
\begin{proof}[Proof of Theorem \ref{MT-3}]
	In \cite{koike}, Koike proved that 
	\begin{align*}
		a_q(E_\lambda)&=-q\cdot\varphi(-1)\cdot{_2}F_{1}\left(\begin{array}{cc}
			\varphi,&\varphi\\
			&\varepsilon
		\end{array}|\lambda\right)_q.
	\end{align*}
For $\lambda\neq0$, using \cite[Proposition 2.5]{mccarthy3} we obtain
\begin{align*}
	a_q(E_\lambda)&=-q\cdot\varphi(-1){\varphi\choose\varepsilon}{_2}F_{1}\left(\begin{array}{cc}
		\varphi,&\varphi\\
		&\varepsilon
	\end{array}|\lambda\right)_q^{*}.
	\end{align*}
Now,  \cite[Lemma 3.3]{mccarthy2} and \eqref{eq-0.4} yield
\begin{align}
	a_q(E_\lambda)&=\varphi(-1)\cdot{_2}G_{2}\left[\begin{array}{cc}
	\frac{1}{2},& \frac{1}{2}\\
	0,&0
\end{array}|\frac{1}{\lambda}\right]_q\label{eq--10}.
\end{align}
Employing Theorem \ref{MT-1} with $a_1=a_2=\frac{1}{2}$ and $a_3=a_4=0$, we have, for $q\equiv1\pmod{2}$,
	\begin{align}\label{eqn-new-05}
		{_2}G_{2}\left[\begin{array}{cc}
			\frac{1}{2},& \frac{1}{2}\\
			0,&0
		\end{array}|\frac{1}{\lambda}\right]_q+{_2}G_{2}\left[\begin{array}{cc}
		\frac{1}{2},& \frac{1}{2}\\
		0,&0
	\end{array}|-\frac{1}{\lambda}\right]_q={_4}G_{4}\left[\begin{array}{cccc}
	\frac{1}{4},& \frac{3}{4},&\frac{1}{4},& \frac{3}{4}\vspace*{0.05cm}\\
	0,&\frac{1}{2},&0,& \frac{1}{2}
\end{array}|\frac{1}{\lambda^2}\right]_q.
\end{align}
Then, \eqref{eq--10} and \eqref{eqn-new-05} yield
\begin{align*}
	a_q(E_\lambda)+a_q(E_{-\lambda})&=\varphi(-1)\cdot{_4}G_{4}\left[\begin{array}{cccc}
		\frac{1}{4},& \frac{3}{4},&\frac{1}{4},& \frac{3}{4}\vspace*{0.05cm}\\
		0,&\frac{1}{2},&0,& \frac{1}{2}
	\end{array}|\frac{1}{\lambda^2}\right]_q\\
&=\varphi(-1)\cdot{_4}G_{4}\left[\begin{array}{cccc}
	0,&\frac{1}{2},&0,& \frac{1}{2}\vspace*{0.05cm}\\
	\frac{1}{4},& \frac{3}{4},&\frac{1}{4},& \frac{3}{4}
\end{array}|\lambda^2\right]_q.
	\end{align*}
This completes the proof of the theorem.
\end{proof}
\begin{proof}[Proof of Theorem \ref{MT-4}]
From the proof of Theorem 1.1 of \cite{lennon}, we have
\begin{align*}
	a_q(E_{a_1,a_3})=-\frac{1}{q}-\frac{1}{q(q-1)}\sum_{l=0}^{q-2}g(T^{-l})^3g(T^{3l})T^l\left(\frac{-a_3}{a_1^3}\right).
\end{align*}
Now, we can write
\begin{align*}
		a_q(E_{a_1,a_3})+a_q(E_{a_1,-a_3})=-\frac{2}{q}-\frac{1}{q(q-1)}\sum_{l=0}^{q-2}g(T^{-l})^3g(T^{3l})T^l\left(\frac{a_3}{a_1^3}\right)(T^l(-1)+1).
\end{align*}
We know that $T^l(-1)=-1$ if $l$ is odd and $T^l(-1)=1$ if $l$ is even. Thus, we deduce that
\begin{align*}
	a_q(E_{a_1,a_3})+a_q(E_{a_1,-a_3})=-\frac{2}{q}-\frac{2}{q(q-1)}\sum_{l=0}^{\frac{q-3}{2}}g(T^{-2l})^3g(T^{6l})T^l\left(\frac{a_3^2}{a_1^6}\right).
\end{align*}
Taking out the term under the summation for $l=0$, we obtain
\begin{align*}
	a_q(E_{a_1,a_3})+a_q(E_{a_1,-a_3})=-\frac{2}{q-1}-\frac{2}{q(q-1)}\sum_{l=1}^{\frac{q-3}{2}}g(T^{-2l})^3g(T^{6l})T^l\left(\frac{a_3^2}{a_1^6}\right).
\end{align*}
Taking $T=\overline{\omega}$ and then using Gross-Koblitz formula, we have
\begin{align*}
	&a_q(E_{a_1,a_3})+a_q(E_{a_1,-a_3})\\
	&=-\frac{2}{q-1}-\frac{2}{q(q-1)}\sum_{l=1}^{\frac{q-3}{2}}\overline{\omega}^l\left(\frac{a_3^2}{a_1^6}\right)(-p)^{\sum_{i=0}^{r-1}\left(3\left\langle-\frac{2lp^i}{q-1}\right\rangle+\left\langle\frac{6lp^i}{q-1}\right\rangle\right)}\\
	&\times\prod_{i=0}^{r-1}\Gamma_p\left(\left\langle-\frac{2lp^i}{q-1}\right\rangle\right)^3\Gamma_p\left(\left\langle\frac{6lp^i}{q-1}\right\rangle\right)\\
	&=-\frac{2}{q-1}-\frac{2}{q(q-1)}\sum_{l=1}^{\frac{q-3}{2}}\overline{\omega}^l\left(\frac{a_3^2}{a_1^6}\right)(-p)^{\sum_{i=0}^{r-1}\left(-3\left\lfloor-\frac{2lp^i}{q-1}\right\rfloor-\left\lfloor\frac{6lp^i}{q-1}\right\rfloor\right)}\\
	&\times\prod_{i=0}^{r-1}\Gamma_p\left(\left\langle-\frac{2lp^i}{q-1}\right\rangle\right)^3\Gamma_p\left(\left\langle\frac{6lp^i}{q-1}\right\rangle\right).
\end{align*}
Using Lemma \ref{lemma-3_1} with $t=2$, Lemma \ref{lemma-3_2} with $t=6$, Lemma \ref{lemma-3_3} with $d=2$, Lemma \ref{lemma-3_4} with $l=6$ and the fact that
\begin{align*}
	\prod_{h=1, h\neq 3}^{5}\Gamma_p\left(\left\langle\frac{hp^i}{6}+\frac{lp^i}{q-1}\right\rangle\right)
    =\prod_{h=1, h\neq 3}^{5}\Gamma_p\left(\left\langle\frac{-hp^i}{6}+\frac{lp^i}{q-1}\right\rangle\right),
	\end{align*}
we have
\begin{align*}
	&a_q(E_{a_1,a_3})+a_q(E_{a_1,-a_3})\\
	&=-\frac{2}{q-1}-\frac{2}{q(q-1)}\sum_{l=1}^{\frac{q-3}{2}}\overline{\omega}^l\left(\frac{a_3^2}{a_1^6}\right)\overline{\omega}^l(3^6)\\
	&\times(-p)^{\sum_{i=0}^{r-1}\left(s_{i,l}-\left\lfloor\left\langle\frac{p^i}{2}\right\rangle-\frac{lp^i}{q-1}\right\rfloor+1+\left\lfloor\frac{lp^i}{q-1}\right\rfloor-\left\lfloor\left\langle\frac{-p^i}{2}\right\rangle+\frac{lp^i}{q-1}\right\rfloor-\left\lfloor\frac{lp^i}{q-1}\right\rfloor\right)}\\
	&\times\prod_{i=0}^{r-1}M_{i,l}\frac{\Gamma_p\left(\left\langle\frac{p^i}{2}-\frac{lp^i}{q-1}\right\rangle\right)\Gamma_p\left(\left\langle\frac{-p^i}{2}+\frac{lp^i}{q-1}\right\rangle\right)
	\Gamma_p\left(\left\langle p^i-\frac{lp^i}{q-1}\right\rangle\right)\Gamma_p\left(\left\langle\frac{lp^i}{q-1}\right\rangle\right)}{\Gamma_p\left(\left\langle\frac{p^i}{2}\right\rangle\right)^2},
\end{align*}
where
\begin{align*}
	M_{i,l}=\frac{\Gamma_p\left(\left\langle\frac{p^i}{2}-\frac{lp^i}{q-1}\right\rangle\right)^2\Gamma_p\left(\left\langle\frac{-lp^i}{q-1}\right\rangle\right)^2\Gamma_p\left(\left\langle\frac{-p^i}{6}+\frac{lp^i}{q-1}\right\rangle\right)
	\Gamma_p\left(\left\langle\frac{-p^i}{3}+\frac{lp^i}{q-1}\right\rangle\right)}{\Gamma_p\left(\left\langle\frac{p^i}{2}\right\rangle\right)^2\Gamma_p\left(\left\langle\frac{-p^i}{6}\right\rangle\right)\Gamma_p\left(\left\langle\frac{-p^i}{3}\right\rangle\right)
	\Gamma_p\left(\left\langle\frac{-2p^i}{3}\right\rangle\right)\Gamma_p\left(\left\langle\frac{-5p^i}{6}\right\rangle\right)}\\
	\times \Gamma_p\left(\left\langle\frac{-2p^i}{3}+\frac{lp^i}{q-1}\right\rangle\right)\Gamma_p\left(\left\langle\frac{-5p^i}{6}+\frac{lp^i}{q-1}\right\rangle\right)
\end{align*}
and
\begin{align*} s_{i,l}&=-2\left\lfloor\left\langle\frac{p^i}{2}\right\rangle-\frac{lp^i}{q-1}\right\rfloor+2+2\left\lfloor\frac{lp^i}{q-1}\right\rfloor-\left\lfloor\left\langle\frac{-p^i}{6}\right\rangle+\frac{lp^i}{q-1}\right\rfloor\\
&-\left\lfloor\left\langle\frac{-p^i}{3}\right\rangle+\frac{lp^i}{q-1}\right\rfloor-\left\lfloor\left\langle\frac{-2p^i}{3}\right\rangle+\frac{lp^i}{q-1}\right\rfloor-\left\lfloor\left\langle\frac{-5p^i}{6}\right\rangle+\frac{lp^i}{q-1}\right\rfloor\\
&=-2\left\lfloor\left\langle\frac{p^i}{2}\right\rangle-\frac{lp^i}{q-1}\right\rfloor-2\left\lfloor\frac{-lp^i}{q-1}\right\rfloor-\left\lfloor\left\langle\frac{-p^i}{6}\right\rangle+\frac{lp^i}{q-1}\right\rfloor-\left\lfloor\left\langle\frac{-p^i}{3}\right\rangle+\frac{lp^i}{q-1}\right\rfloor\\
&-\left\lfloor\left\langle\frac{-2p^i}{3}\right\rangle+\frac{lp^i}{q-1}\right\rfloor-\left\lfloor\left\langle\frac{-5p^i}{6}\right\rangle+\frac{lp^i}{q-1}\right\rfloor.
\end{align*}
Using \eqref{eq-12}, \eqref{eq-13} and the fact that $\left\lfloor\left\langle\frac{p^i}{2}\right\rangle-\frac{lp^i}{q-1}\right\rfloor+\left\lfloor\left\langle\frac{-p^i}{2}\right\rangle+\frac{lp^i}{q-1}\right\rfloor=0$ for $l\neq\frac{q-1}{2}$, we have
\begin{align*}
	a_q(E_{a_1,a_3})+a_q(E_{a_1,-a_3})&=-\frac{2}{q-1}-\frac{2}{q-1}\sum_{l=1}^{\frac{q-3}{2}}\overline{\omega}^l\left(\frac{729a_3^2}{a_1^6}\right)(-p)^{\sum_{i=0}^{r-1}s_{i,l}}\prod_{i=0}^{r-1}M_{i,l}\\
	&=-\frac{2}{q-1}\sum_{l=0}^{\frac{q-3}{2}}\overline{\omega}^l\left(\frac{729a_3^2}{a_1^6}\right)(-p)^{\sum_{i=0}^{r-1}s_{i,l}}\prod_{i=0}^{r-1}M_{i,l}\\
	&=-\frac{1}{q-1}\sum_{l=0}^{\frac{q-3}{2}}\overline{\omega}^l\left(\frac{729a_3^2}{a_1^6}\right)(-p)^{\sum_{i=0}^{r-1}s_{i,l}}\prod_{i=0}^{r-1}M_{i,l}\\
	&-\frac{1}{q-1}\sum_{l=0}^{\frac{q-3}{2}}\overline{\omega}^l\left(\frac{729a_3^2}{a_1^6}\right)(-p)^{\sum_{i=0}^{r-1}s_{i,l}}\prod_{i=0}^{r-1}M_{i,l}\\	&=-\frac{1}{q-1}\sum_{l=0}^{\frac{q-3}{2}}\overline{\omega}^l\left(\frac{729a_3^2}{a_1^6}\right)(-p)^{\sum_{i=0}^{r-1}s_{i,l}}\prod_{i=0}^{r-1}M_{i,l}\\
	&-\frac{1}{q-1}\sum_{l=\frac{q-1}{2}}^{q-2}\overline{\omega}^l\left(\frac{729a_3^2}{a_1^6}\right)(-p)^{\sum_{i=0}^{r-1}s_{i,l-\frac{q-1}{2}}}\prod_{i=0}^{r-1}M_{i,l-\frac{q-1}{2}}.
\end{align*}
We can easily check that $s_{i,l-\frac{q-1}{2}}=s_{i,l}$ and $M_{i,l-\frac{q-1}{2}}=M_{i,l}$. Hence, we obtain 
\begin{align*}
	a_q(E_{a_1,a_3})+a_q(E_{a_1,-a_3})&=-\frac{1}{q-1}\sum_{l=0}^{q-2}\overline{\omega}^l\left(\frac{729a_3^2}{a_1^6}\right)(-p)^{\sum_{i=0}^{r-1}s_{i,l}}\prod_{i=0}^{r-1}M_{i,l}\\
	&={_4}G_4\left[\begin{array}{cccc}
		0, & \frac{1}{2}, & 0, & \frac{1}{2}\vspace*{0.1cm}\\
		\frac{1}{6}, & \frac{1}{3}, & \frac{2}{3}, & \frac{5}{6}
	\end{array}|\frac{729a_3^2}{a_1^6}
	\right]_q.
\end{align*}
This completes the proof of the theorem.
\end{proof}
\begin{proof}[Proof of Theorem \ref{MT-5}]
From \cite[(3.8)]{BK}, we have 
\begin{align}\label{eq--12}
q(\#E_{f,g}(\mathbb{F}_q)-1)=q^2+\frac{g(\varphi)\varphi(-f)}{q-1}\sum_{l=0}^{q-2}g(T^{-l})g(T^{2l+\frac{q-1}{2}})g(T^{-l})T^{-2l}(f)T^l(g).
\end{align} 
Note that in \cite[(3.8)]{BK}, we have $\varphi(f)$ in place of $\varphi(-f)$. This is because in \cite[(3.8)]{BK} it is assumed that $q\equiv 1\pmod{4}$, and hence $\varphi(f)=\varphi(-f)$.
Using Davenport-Hasse relation with $m=2$ and $\psi=T^{2l}$, we have 
\begin{align}\label{eqn-aa}
	g(T^{2l+\frac{q-1}{2}})=\frac{g(T^{4l})g(\varphi)T^{-l}(16)}{g(T^{2l})}.
\end{align}
Substituting \eqref{eqn-aa} in \eqref{eq--12}, we obtain
\begin{align*}
	q(\#E_{f,g}(\mathbb{F}_q)-1)&=q^2+\frac{g(\varphi)^2\varphi(-f)}{q-1}\sum_{l=0}^{q-2}\frac{g(T^{-l})^2g(T^{4l})}{g(T^{2l})}T^l\left(\frac{g}{16f^2}\right)\\
	&=q^2+\frac{q\varphi(f)}{q-1}\sum_{l=0}^{q-2}\frac{g(T^{-l})^2g(T^{4l})}{g(T^{2l})}T^l\left(\frac{g}{16f^2}\right),
\end{align*}
 where the last equality follows from Lemma \ref{lemma2_1}.
Using the relation $a_q(E_{f,g})=q+1-\#E_{f,g}(\mathbb{F}_q)$, we have
\begin{align*}
	a_q(E_{f,g})=-\frac{\varphi(f)}{q-1}\sum_{l=0}^{q-2}\frac{g(T^{-l})^2g(T^{4l})}{g(T^{2l})}T^l\left(\frac{g}{16f^2}\right).
\end{align*}
Hence, we obtain
\begin{align*}
	a_q(E_{f,g})+a_q(E_{f,-g})=-\frac{\varphi(f)}{q-1}\sum_{l=0}^{q-2}\frac{g(T^{-l})^2g(T^{4l})}{g(T^{2l})}T^l\left(\frac{g}{16f^2}\right)(1+T^l(-1)).
\end{align*}
We know that $T^l(-1)=1$ if $l$ is even and $T^l(-1)=-1$ if $l$ is odd. Thus, we obtain
\begin{align*}
	a_q(E_{f,g})+a_q(E_{f,-g})=-\frac{2\varphi(f)}{q-1}\sum_{l=0}^{\frac{q-3}{2}}\frac{g(T^{-2l})^2g(T^{8l})}{g(T^{4l})}T^l\left(\frac{g^2}{16^2f^4}\right).
\end{align*}
Taking $T=\overline{\omega}$ and then using Gross-Koblitz formula, we obtain
\begin{align*}
	a_q(E_{f,g})+a_q(E_{f,-g})&=-\frac{2\varphi(f)}{q-1}\sum_{l=0}^{\frac{q-3}{2}}\overline{\omega}^l\left(\frac{g^2}{16^2f^4}\right)(-p)^{\sum_{i=0}^{r-1}\left(2\left\langle-\frac{2lp^i}{q-1}\right\rangle+\left\langle\frac{8lp^i}{q-1}\right\rangle-\left\langle\frac{4lp^i}{q-1}\right\rangle\right)}\\
	&\hspace*{1cm}\times\prod_{i=0}^{r-1}\frac{\Gamma_p\left(\left\langle-\frac{2lp^i}{q-1}\right\rangle\right)^2\Gamma_p\left(\left\langle\frac{8lp^i}{q-1}\right\rangle\right)}{\Gamma_p\left(\left\langle\frac{4lp^i}{q-1}\right\rangle\right)}\\
    &=-\frac{2\varphi(f)}{q-1}\sum_{l=0}^{\frac{q-3}{2}}\overline{\omega}^l\left(\frac{g^2}{16^2f^4}\right)(-p)^{\sum_{i=0}^{r-1}\left(-2\left\lfloor-\frac{2lp^i}{q-1}\right\rfloor-\left\lfloor\frac{8lp^i}{q-1}\right\rfloor+\left\lfloor\frac{4lp^i}{q-1}\right\rfloor\right)}\\
	&\hspace*{1cm}\times\prod_{i=0}^{r-1}\frac{\Gamma_p\left(\left\langle-\frac{2lp^i}{q-1}\right\rangle\right)^2\Gamma_p\left(\left\langle\frac{8lp^i}{q-1}\right\rangle\right)}{\Gamma_p\left(\left\langle\frac{4lp^i}{q-1}\right\rangle\right)}.
\end{align*}
Taking out the term under the summation for $l=0$ and then using Lemma \ref{lemma-3_1} with $t=2$, Lemma \ref{lemma-3_2} with $t=8$ and $t=4$, Lemma \ref{lemma-3_3} with $d=2$, Lemma \ref{lemma-3_4} with $l=8$ and $l=4$ and the fact that 
\begin{align*}
	\prod_{h=0}^{3}\Gamma_p\left(\left\langle\frac{(2h+1)p^i}{8}+\frac{lp^i}{q-1}\right\rangle\right)
	=\prod_{h=0}^{3}\Gamma_p\left(\left\langle\frac{-(2h+1)p^i}{8}+\frac{lp^i}{q-1}\right\rangle\right),
	\end{align*}
we obtain
\begin{align*}
a_q(E_{f,g})+a_q(E_{f,-g})&=-\frac{2\varphi(f)}{q-1}-\frac{2\varphi(f)}{q-1}\sum_{l=1}^{\frac{q-3}{2}}\overline{\omega}^l\left(\frac{16g^2}{f^4}\right)(-p)^{\sum_{i=0}^{r-1}\alpha_{i,l}}\prod_{i=0}^{r-1}N_{i,l},
\end{align*}
where 
\begin{align*}
	N_{i,l}=\frac{\Gamma_p\left(\left\langle\frac{p^i}{2}-\frac{lp^i}{q-1}\right\rangle\right)^2\Gamma_p\left(\left\langle\frac{-lp^i}{q-1}\right\rangle\right)^2\Gamma_p\left(\left\langle\frac{-p^i}{8}+\frac{lp^i}{q-1}\right\rangle\right)\Gamma_p\left(\left\langle\frac{-3p^i}{8}+\frac{lp^i}{q-1}\right\rangle\right)}{\Gamma_p\left(\left\langle\frac{p^i}{2}\right\rangle\right)^2\Gamma_p\left(\left\langle\frac{-p^i}{8}\right\rangle\right)\Gamma_p\left(\left\langle\frac{-3p^i}{8}\right\rangle\right)\Gamma_p\left(\left\langle\frac{-5p^i}{8}\right\rangle\right)\Gamma_p\left(\left\langle\frac{-7p^i}{8}\right\rangle\right)}\\
	\times \Gamma_p\left(\left\langle\frac{-5p^i}{8}+\frac{lp^i}{q-1}\right\rangle\right)\Gamma_p\left(\left\langle\frac{-7p^i}{8}+\frac{lp^i}{q-1}\right\rangle\right)
\end{align*}
and
\begin{align*} \alpha_{i,l}&=-2\left\lfloor\left\langle\frac{p^i}{2}\right\rangle-\frac{lp^i}{q-1}\right\rfloor-2\left\lfloor\frac{-lp^i}{q-1}\right\rfloor-\left\lfloor\left\langle\frac{-p^i}{8}\right\rangle+\frac{lp^i}{q-1}\right\rfloor\\
&-\left\lfloor\left\langle\frac{-3p^i}{8}\right\rangle+\frac{lp^i}{q-1}\right\rfloor-\left\lfloor\left\langle\frac{-5p^i}{8}\right\rangle+\frac{lp^i}{q-1}\right\rfloor-\left\lfloor\left\langle\frac{-7p^i}{8}\right\rangle+\frac{lp^i}{q-1}\right\rfloor.
\end{align*} 
Thus, we can write
\begin{align*}
	a_q(E_{f,g})+a_q(E_{f,-g})&=-\frac{2\varphi(f)}{q-1}\sum_{l=0}^{\frac{q-3}{2}}\overline{\omega}^l\left(\frac{16g^2}{f^4}\right)(-p)^{\sum_{i=0}^{r-1}\alpha_{i,l}}\prod_{i=0}^{r-1}N_{i,l}\\
	&=-\frac{\varphi(f)}{q-1}\sum_{l=0}^{\frac{q-3}{2}}\overline{\omega}^l\left(\frac{16g^2}{f^4}\right)(-p)^{\sum_{i=0}^{r-1}\alpha_{i,l}}\prod_{i=0}^{r-1}N_{i,l}\\
	&\hspace*{1cm}-\frac{\varphi(f)}{q-1}\sum_{l=0}^{\frac{q-3}{2}}\overline{\omega}^l\left(\frac{16g^2}{f^4}\right)(-p)^{\sum_{i=0}^{r-1}\alpha_{i,l}}\prod_{i=0}^{r-1}N_{i,l}\\
	&=-\frac{\varphi(f)}{q-1}\sum_{l=0}^{\frac{q-3}{2}}\overline{\omega}^l\left(\frac{16g^2}{f^4}\right)(-p)^{\sum_{i=0}^{r-1}\alpha_{i,l}}\prod_{i=0}^{r-1}N_{i,l}\\
	&\hspace*{.2cm}-\frac{\varphi(f)}{q-1}\sum_{l=\frac{q-1}{2}}^{q-2}\overline{\omega}^l\left(\frac{16g^2}{f^4}\right)(-p)^{\sum_{i=0}^{r-1}\alpha_{i,l-\frac{q-1}{2}}}\prod_{i=0}^{r-1}N_{i,l-\frac{q-1}{2}}.
\end{align*}
It is easy to see that $N_{i,l-\frac{q-1}{2}}=N_{i,l}$ and $\alpha_{i,l-\frac{q-1}{2}}=\alpha_{i,l}$. Hence, we obtain
\begin{align*}
	a_q(E_{f,g})+a_q(E_{f,-g})&=-\frac{\varphi(f)}{q-1}\sum_{l=0}^{q-2}\overline{\omega}^l\left(\frac{16g^2}{f^4}\right)(-p)^{\sum_{i=0}^{r-1}\alpha_{i,l}}\prod_{i=0}^{r-1}N_{i,l}\\
	&=\varphi(f)\cdot{_4}G_4\left[\begin{array}{cccc}
			0, & \frac{1}{2}, & 0, & \frac{1}{2}\vspace*{0.1cm}\\
			\frac{1}{8}, & \frac{3}{8}, & \frac{5}{8}, & \frac{7}{8}
		\end{array}|\frac{16g^{2}}{f^4}
		\right]_q.
\end{align*}
This completes the proof of the theorem.
	\end{proof}
\begin{proof}[Proof of Theorem \ref{MT-6.0}]
	 From \cite[(3.4)]{BK}, we obtain
	\begin{align}\label{eq--15}
		q(\#E_{c,d}(\mathbb{F}_q)-1)=q^2-1+B+C,
	\end{align} 
	where
	\begin{align*}
		&B=1+q\varphi(d),\\
		&C=\frac{\varphi(-c)g(\varphi)}{q-1}\sum_{n=0}^{q-2}g(T^{-2n})g(T^{3n+\frac{q-1}{2}})g(T^{-n})T^n\left(\frac{d}{c^3}\right).
	\end{align*}
	Using Davenport-Hasse relation with $m=2$ and $\psi=T^{3n}$, we have
	\begin{align}\label{eq--14}
		g(T^{3n+\frac{q-1}{2}})=\frac{g(T^{6n})g(\varphi)T^{-3n}(4)}{g(T^{3n})}.
	\end{align}
	Substituting \eqref{eq--14} in the expression of $C$ yields
	\begin{align*}
		C&=\frac{\varphi(-c)g(\varphi)}{q-1}\sum_{n=0}^{q-2}\frac{g(T^{-n})g(T^{-2n})g(T^{6n})g(\varphi)}{g(T^{3n})}T^n\left(\frac{d}{4^3c^3}\right)\\
		&=\frac{q\varphi(c)}{q-1}\sum_{n=0}^{q-2}\frac{g(T^{-n})g(T^{-2n})g(T^{6n})}{g(T^{3n})}T^n\left(\frac{d}{4^3c^3}\right),
	\end{align*}
	where the last equality follows from Lemma \ref{lemma2_1}. Now, substituting the values of $B$ and $C$ in \eqref{eq--15}, we have
	\begin{align*}
		q(\#E_{c,d}(\mathbb{F}_q)-1)=q^2+q\varphi(d)+\frac{q\varphi(c)}{q-1}\sum_{n=0}^{q-2}\frac{g(T^{-n})g(T^{-2n})g(T^{6n})}{g(T^{3n})}T^n\left(\frac{d}{4^3c^3}\right).
	\end{align*}
	Using the relation $a_q(E_{c,d})=q+1-\#E_{c,d}({\mathbb{F}_q})$, we obtain
	\begin{align*}
		a_q(E_{c,d})=-\varphi(d)-\frac{\varphi(c)}{q-1}\sum_{n=0}^{q-2}\frac{g(T^{-n})g(T^{-2n})g(T^{6n})}{g(T^{3n})}T^n\left(\frac{d}{4^3c^3}\right).
	\end{align*}
	Hence, we have 
	\begin{align*}
		a_q(E_{c,d})+a_q(E_{c,-d})&=-\frac{\varphi(c)}{q-1}\sum_{n=0}^{q-2}\frac{g(T^{-n})g(T^{-2n})g(T^{6n})}{g(T^{3n})}T^n\left(\frac{d}{4^3c^3}\right)(T^n(-1)+1)\\
		&-\varphi(d)-\varphi(-d).
	\end{align*}
	Using the fact that $T^{n}(-1)=-1$ if $n$ is odd and $T^n(-1)=1$ if $n$ is even, we deduce that
	\begin{align*}
		a_q(E_{c,d})+a_q(E_{c,-d})&=-\frac{2\varphi(c)}{q-1}\sum_{n=0}^{\frac{q-3}{2}}\frac{g(T^{-2n})g(T^{-4n})g(T^{12n})}{g(T^{6n})}T^n\left(\frac{d^2}{4^6c^6}\right)\\
		&-\varphi(d)-\varphi(-d).
	\end{align*}
	Taking $T=\overline{\omega}$ and then using Gross-Koblitz formula, we obtain
	\begin{align*}
		&a_q(E_{c,d})+a_q(E_{c,-d})\\
		&=-\frac{2\varphi(c)}{q-1}\sum_{n=0}^{\frac{q-3}{2}}(-p)^{\sum_{i=0}^{r-1}\left(\left\langle-\frac{2np^i}{q-1}\right\rangle+\left\langle-\frac{4np^i}{q-1}\right\rangle+\left\langle\frac{12np^i}{q-1}\right\rangle-\left\langle\frac{6np^i}{q-1}\right\rangle\right)}\overline{\omega}^{n}\left(\frac{d^2}{4^6c^6}\right)\\
		&\times\prod_{i=0}^{r-1}\frac{\Gamma_p\left(\left\langle-\frac{2np^i}{q-1}\right\rangle\right)\Gamma_p\left(\left\langle-\frac{4np^i}{q-1}\right\rangle\right)\Gamma_p\left(\left\langle\frac{12np^i}{q-1}\right\rangle\right)}{\Gamma_p\left(\left\langle\frac{6np^i}{q-1}\right\rangle\right)}-\varphi(d)-\varphi(-d)\\
		&=-\frac{2\varphi(c)}{q-1}-\frac{2\varphi(c)}{q-1}\sum_{n=1}^{\frac{q-3}{2}}\overline{\omega}^{n}\left(\frac{d^2}{4^6c^6}\right)(-p)^{\sum_{i=0}^{r-1}\left(-\left\lfloor-\frac{2np^i}{q-1}\right\rfloor-\left\lfloor-\frac{4np^i}{q-1}\right\rfloor-\left\lfloor\frac{12np^i}{q-1}\right\rfloor+\left\lfloor\frac{6np^i}{q-1}\right\rfloor\right)}\\
		&\times \prod_{i=0}^{r-1}\frac{\Gamma_p\left(\left\langle-\frac{2np^i}{q-1}\right\rangle\right)\Gamma_p\left(\left\langle-\frac{4np^i}{q-1}\right\rangle\right)\Gamma_p\left(\left\langle\frac{12np^i}{q-1}\right\rangle\right)}{\Gamma_p\left(\left\langle\frac{6np^i}{q-1}\right\rangle\right)}-\varphi(d)-\varphi(-d).
	\end{align*}
	Using Lemma \ref{lemma-3_1} with $t=2$ and $t=4$, Lemma \ref{lemma-3_2} with $t=12$ and $t=6$, Lemma \ref{lemma-3_3} with $d=2$ and $d=4$, Lemma \ref{lemma-3_4} with $l=12$ and $l=6$, and the fact that
	\begin{align*}
	\prod_{h=0}^{5}\Gamma_p\left(\left\langle\frac{(2h+1)p^i}{12}+\frac{np^i}{q-1}\right\rangle\right)
	=\prod_{h=0}^{5}\Gamma_p\left(\left\langle\frac{-(2h+1)p^i}{12}+\frac{np^i}{q-1}\right\rangle\right),
		\end{align*}
	we obtain
	\begin{align}\label{eqn--26}
		a_q(E_{c,d})+a_q(E_{c,-d})&=-\frac{2\varphi(c)}{q-1}\sum_{n=1}^{\frac{q-3}{2}}\overline{\omega}^{n}\left(\frac{729d^2}{16c^6}\right)(-p)^{\sum_{i=0}^{r-1}\left(\beta_{i,n}+\gamma_{i,n}\right)}\prod_{i=0}^{r-1}A_{i,n}B_{i,n}\notag\\
		&-\varphi(d)-\varphi(-d)-\frac{2\varphi(c)}{q-1},
	\end{align}
	where
	\begin{align*}
		\beta_{i,n}&=-2\left\lfloor\left\langle\frac{p^i}{2}\right\rangle-\frac{np^i}{q-1}\right\rfloor-2\left\lfloor\frac{-np^i}{q-1}\right\rfloor-\left\lfloor\left\langle\frac{-p^i}{12}\right\rangle+\frac{np^i}{q-1}\right\rfloor
		\\
		&\hspace{.4cm} -\left\lfloor\left\langle\frac{-5p^i}{12}\right\rangle+\frac{np^i}{q-1}\right\rfloor-\left\lfloor\left\langle\frac{-7p^i}{12}\right\rangle+\frac{np^i}{q-1}\right\rfloor-\left\lfloor\left\langle\frac{-11p^i}{12}\right\rangle+\frac{np^i}{q-1}\right\rfloor,\\
		\gamma_{i,n}&=-\left\lfloor\left\langle\frac{-p^i}{4}\right\rangle+\frac{np^i}{q-1}\right\rfloor-\left\lfloor\left\langle\frac{3p^i}{4}\right\rangle-\frac{np^i}{q-1}\right\rfloor-\left\lfloor\left\langle\frac{p^i}{4}\right\rangle-\frac{np^i}{q-1}\right\rfloor\\
		&\hspace{.4cm}-\left\lfloor\left\langle\frac{-3p^i}{4}\right\rangle+\frac{np^i}{q-1}\right\rfloor,\\
		A_{i,n}&=\frac{\Gamma_p\left(\left\langle\frac{p^i}{2}-\frac{np^i}{q-1}\right\rangle\right)^2\Gamma_p\left(\left\langle\frac{-np^i}{q-1}\right\rangle\right)^2\Gamma_p\left(\left\langle\frac{-p^i}{12}+\frac{np^i}{q-1}\right\rangle\right)\Gamma_p\left(\left\langle\frac{-5p^i}{12}+\frac{np^i}{q-1}\right\rangle\right)}{\Gamma_p\left(\left\langle\frac{p^i}{2}\right\rangle\right)^2\Gamma_p\left(\left\langle\frac{-p^i}{12}\right\rangle\right)\Gamma_p\left(\left\langle\frac{-5p^i}{12}\right\rangle\right)\Gamma_p\left(\left\langle\frac{-7p^i}{12}\right\rangle\right)\Gamma_p\left(\left\langle\frac{-11p^i}{12}\right\rangle\right)}\\
		&\times\Gamma_p\left(\left\langle\frac{-7p^i}{12}+\frac{np^i}{q-1}\right\rangle\right)\Gamma_p\left(\left\langle\frac{-11p^i}{12}+\frac{np^i}{q-1}\right\rangle\right),\\
		B_{i,n}&=\frac{\Gamma_p\left(\left\langle\frac{-p^i}{4}+\frac{np^i}{q-1}\right\rangle\right)\Gamma_p\left(\left\langle\frac{-3p^i}{4}+\frac{np^i}{q-1}\right\rangle\right)\Gamma_p\left(\left\langle\frac{p^i}{4}-\frac{np^i}{q-1}\right\rangle\right)\Gamma_p\left(\left\langle\frac{3p^i}{4}-\frac{np^i}{q-1}\right\rangle\right)}{\Gamma_p\left(\left\langle\frac{p^i}{4}\right\rangle\right)^2\Gamma_p\left(\left\langle\frac{3p^i}{4}\right\rangle\right)^2}.
	\end{align*}
Adding and subtracting the term under the summation for $n=0$ yields
\begin{align*}
	&a_q(E_{c,d})+a_q(E_{c,-d})\\
	&=-\frac{2\varphi(c)}{q-1}\sum_{n=0}^{\frac{q-3}{2}}\overline{\omega}^{n}\left(\frac{729d^2}{16c^6}\right)(-p)^{\sum_{i=0}^{r-1}\left(\beta_{i,n}+\gamma_{i,n}\right)}\prod_{i=0}^{r-1}A_{i,n}B_{i,n}\notag\\
	&-\varphi(d)-\varphi(-d)\\
	&=-\frac{\varphi(c)}{q-1}\sum_{n=0}^{\frac{q-3}{2}}\overline{\omega}^{n}\left(\frac{729d^2}{16c^6}\right)(-p)^{\sum_{i=0}^{r-1}\left(\beta_{i,n}+\gamma_{i,n}\right)}\prod_{i=0}^{r-1}A_{i,n}B_{i,n}-\varphi(d)-\varphi(-d)\\
	&-\frac{\varphi(c)}{q-1}\sum_{n=0}^{\frac{q-3}{2}}\overline{\omega}^{n}\left(\frac{729d^2}{16c^6}\right)(-p)^{\sum_{i=0}^{r-1}\left(\beta_{i,n}+\gamma_{i,n}\right)}\prod_{i=0}^{r-1}A_{i,n}B_{i,n}\\
	&=-\frac{\varphi(c)}{q-1}\sum_{n=0}^{\frac{q-3}{2}}\overline{\omega}^{n}\left(\frac{729d^2}{16c^6}\right)(-p)^{\sum_{i=0}^{r-1}\left(\beta_{i,n}+\gamma_{i,n}\right)}\prod_{i=0}^{r-1}A_{i,n}B_{i,n}-\varphi(d)-\varphi(-d)\\
	&-\frac{\varphi(c)}{q-1}\sum_{n=\frac{q-1}{2}}^{q-2}\overline{\omega}^{n}\left(\frac{729d^2}{16c^6}\right)(-p)^{\sum_{i=0}^{r-1}\left(\beta_{i,n-\frac{q-1}{2}}+\gamma_{i,n-\frac{q-1}{2}}\right)}\prod_{i=0}^{r-1}A_{i,n-\frac{q-1}{2}}B_{i,n-\frac{q-1}{2}}.
\end{align*}
We can easily check that $A_{i,n-\frac{q-1}{2}}=A_{i,n},B_{i,n-\frac{q-1}{2}}=B_{i,n},\beta_{i,n-\frac{q-1}{2}}=\beta_{i,n}$, and $\gamma_{i,n-\frac{q-1}{2}}=\gamma_{i,n}$. Thus, we can write
\begin{align*}
		a_q(E_{c,d})+a_q(E_{c,-d})&=-\frac{\varphi(c)}{q-1}\sum_{n=0}^{q-2}\overline{\omega}^{n}\left(\frac{729d^2}{16c^6}\right)(-p)^{\sum_{i=0}^{r-1}\left(\beta_{i,n}+\gamma_{i,n}\right)}\prod_{i=0}^{r-1}A_{i,n}B_{i,n}\\
		&-\varphi(d)-\varphi(-d)\\
		&=\varphi(c)\cdot{_6}G_6\left[\begin{array}{ccccccc}
			0,& \frac{1}{2}, & 0,& \frac{1}{2},& \frac{1}{4},& \frac{3}{4}\vspace*{0.05cm}\\
			\frac{1}{12},& \frac{1}{4},& \frac{5}{12},& \frac{7}{12},& \frac{3}{4},& \frac{11}{12}
		\end{array}|\frac{729d^2}{16c^6}\right]_q\\
	&-\varphi(d)-\varphi(-d).
\end{align*}
This completes the proof of the theorem.
\end{proof}
\begin{proof}[Proof of Theorem \ref{MT-6}]
	Part (1): Here, $q\equiv1, 7\pmod {12}$. Let $\chi_6$ be a character of order $6$.  Then, by \cite[Theorem 3.1]{BK}, we have 
	\begin{align*}
		a_q(E_{c,d})&=-q\cdot\varphi(-3c)\cdot{_2}F_1\left(\begin{array}{cc}
			\chi_{6}, & \chi_6^5\\
			&\varepsilon
		\end{array}|-\frac{27d}{4c^3}\right)_q\\
	 &=-q\cdot\varphi(-3c){\chi_6^5\choose\varepsilon}{_2}F_1\left(\begin{array}{cc}
	 	\chi_{6}, & \chi_6^5\\
	 	&\varepsilon
	 \end{array}|-\frac{27d}{4c^3}\right)_q^*\\
 &=\varphi(-3c)\cdot{_2}G_2\left[\begin{array}{cc}
 	\frac{1}{6}, & \frac{5}{6}\\
 	0,&0
 \end{array}|-\frac{4c^3}{27d}\right]_q\\
&=\varphi(-3c)\cdot{_2}G_2\left[\begin{array}{cc}
	0,&0\\
	\frac{1}{6}, & \frac{5}{6}
\end{array}|-\frac{27d}{4c^3}\right]_q.
	\end{align*}
Hence, 
\begin{align*}
	&a_q(E_{c,d})+a_q(E_{c,-d})\\
	&=\varphi(-3c){_2}G_2\left[\begin{array}{cc}
		0,&0\\
		\frac{1}{6}, & \frac{5}{6}
	\end{array}|-\frac{27d}{4c^3}\right]_q+\varphi(-3c){_2}G_2\left[\begin{array}{cc}
	0,&0\\
	\frac{1}{6}, & \frac{5}{6}
\end{array}|\frac{27d}{4c^3}\right]_q.
\end{align*}
Using Theorem \ref{MT-1} with $a_1=a_2=0$, $a_3=\frac{1}{6}$, and $a_4=\frac{5}{6}$, we deduce that
\begin{align*}
	a_q(E_{c,d})+a_q(E_{c,-d})=\varphi(-3c){_4}G_{4}\left[\begin{array}{cccc}
		0,&\frac{1}{2}, &0,& \frac{1}{2}\vspace*{0.05cm}\\
		\frac{1}{12}, & \frac{5}{12}, &\frac{7}{12}, &\frac{11}{12}
	\end{array}|\frac{729d^2}{16c^6}\right].
\end{align*}
This completes the proof of (1).\\
Part (2): Here, $q\equiv5\pmod{12}$. From \eqref{eqn--26}, we have
\begin{align}
		a_q(E_{c,d})+a_q(E_{c,-d})&=-\frac{2\varphi(c)}{q-1}\sum_{n=1}^{\frac{q-3}{2}}\overline{\omega}^{n}\left(\frac{729d^2}{16c^6}\right)(-p)^{\sum_{i=0}^{r-1}\left(\beta_{i,n}+\gamma_{i,n}\right)}\prod_{i=0}^{r-1}A_{i,n}B_{i,n}\notag\\
		&-\varphi(d)-\varphi(-d)-\frac{2\varphi(c)}{q-1}\notag,
\end{align}
where
\begin{align*}
\beta_{i,n}&=-2\left\lfloor\left\langle\frac{p^i}{2}\right\rangle-\frac{np^i}{q-1}\right\rfloor-2\left\lfloor\frac{-np^i}{q-1}\right\rfloor-\left\lfloor\left\langle\frac{-p^i}{12}\right\rangle+\frac{np^i}{q-1}\right\rfloor\\
&\hspace{.4cm}-\left\lfloor\left\langle\frac{-5p^i}{12}\right\rangle+\frac{np^i}{q-1}\right\rfloor-\left\lfloor\left\langle\frac{-7p^i}{12}\right\rangle+\frac{np^i}{q-1}\right\rfloor-\left\lfloor\left\langle\frac{-11p^i}{12}\right\rangle+\frac{np^i}{q-1}\right\rfloor,\\
\gamma_{i,n}&=-\left\lfloor\left\langle\frac{-p^i}{4}\right\rangle+\frac{np^i}{q-1}\right\rfloor-\left\lfloor\left\langle\frac{3p^i}{4}\right\rangle-\frac{np^i}{q-1}\right\rfloor-\left\lfloor\left\langle\frac{p^i}{4}\right\rangle-\frac{np^i}{q-1}\right\rfloor\\
&\hspace{.4cm}-\left\lfloor\left\langle\frac{-3p^i}{4}\right\rangle+\frac{np^i}{q-1}\right\rfloor,\\
A_{i,n}&=\frac{\Gamma_p\left(\left\langle\frac{p^i}{2}-\frac{np^i}{q-1}\right\rangle\right)^2\Gamma_p\left(\left\langle\frac{-np^i}{q-1}\right\rangle\right)^2\Gamma_p\left(\left\langle\frac{-p^i}{12}+\frac{np^i}{q-1}\right\rangle\right)\Gamma_p\left(\left\langle\frac{-5p^i}{12}+\frac{np^i}{q-1}\right\rangle\right)}{\Gamma_p\left(\left\langle\frac{p^i}{2}\right\rangle\right)^2\Gamma_p\left(\left\langle\frac{-p^i}{12}\right\rangle\right)\Gamma_p\left(\left\langle\frac{-5p^i}{12}\right\rangle\right)\Gamma_p\left(\left\langle\frac{-7p^i}{12}\right\rangle\right)\Gamma_p\left(\left\langle\frac{-11p^i}{12}\right\rangle\right)}\\
&\times\Gamma_p\left(\left\langle\frac{-7p^i}{12}+\frac{np^i}{q-1}\right\rangle\right)\Gamma_p\left(\left\langle\frac{-11p^i}{12}+\frac{np^i}{q-1}\right\rangle\right),\\
B_{i,n}&=\frac{\Gamma_p\left(\left\langle\frac{-p^i}{4}+\frac{np^i}{q-1}\right\rangle\right)\Gamma_p\left(\left\langle\frac{-3p^i}{4}+\frac{np^i}{q-1}\right\rangle\right)\Gamma_p\left(\left\langle\frac{p^i}{4}-\frac{np^i}{q-1}\right\rangle\right)\Gamma_p\left(\left\langle\frac{3p^i}{4}-\frac{np^i}{q-1}\right\rangle\right)}{\Gamma_p\left(\left\langle\frac{p^i}{4}\right\rangle\right)^2\Gamma_p\left(\left\langle\frac{3p^i}{4}\right\rangle\right)^2}.
\end{align*} 
We can easily check that
\begin{align*}
	\beta_{i,\frac{q-1}{4}}+\gamma_{i,\frac{q-1}{4}}&=0,\\
	\prod_{i=0}^{r-1}A_{i,\frac{q-1}{4}}B_{i,\frac{q-1}{4}}&=\prod_{i=0}^{r-1}\frac{\Gamma_p\left(\left\langle\frac{p^i}{6}\right\rangle\right)\Gamma_p\left(\left\langle\frac{5p^i}{6}\right\rangle\right)\Gamma_p\left(\left\langle\frac{p^i}{3}\right\rangle\right)\Gamma_p\left(\left\langle\frac{2p^i}{3}\right\rangle\right)}{\Gamma_p\left(\left\langle\frac{p^i}{12}\right\rangle\right)\Gamma_p\left(\left\langle\frac{5p^i}{12}\right\rangle\right)\Gamma_p\left(\left\langle\frac{7p^i}{12}\right\rangle\right)\Gamma_p\left(\left\langle\frac{11p^i}{12}\right\rangle\right)}\\
	&=\prod_{i=0}^{r-1}\frac{\varphi(6)\Gamma_p\left(\left\langle\frac{p^i}{4}\right\rangle\right)\Gamma_p\left(\left\langle\frac{3p^i}{4}\right\rangle\right)}{\Gamma_p\left(\left\langle\frac{p^i}{2}\right\rangle\right)^2}\\
	&=\varphi(6)\overline{\omega}^\frac{q-1}{4}(-1),
\end{align*}
where we used Lemma \ref{lemma-3_2} with $t=6$ and $a=\frac{q-1}{4}$; and \eqref{eq-12} with $a=\frac{q-1}{4}$ and $a=\frac{q-1}{2}$ to deduce the last two equations, respectively. Taking out the term under the summation for $n=\frac{q-1}{4}$ and then using Lemma \ref{lemma-13} yield
\begin{align*}
	a_q(E_{c,d})+a_q(E_{c,-d})&=-\frac{2\varphi(c)}{q-1}\sum_{n=1,n\neq\frac{q-1}{4}}^{\frac{q-3}{2}}\overline{\omega}^{n}\left(\frac{729d^2}{16c^6}\right)(-p)^{\sum_{i=0}^{r-1}\beta_{i,n}}\prod_{i=0}^{r-1}A_{i,n}\\
	&-\varphi(d)-\varphi(-d)-\frac{2\varphi(c)}{q-1}-\frac{2\varphi(2d)}{q-1}\overline{\omega}^\frac{q-1}{4}(-1).
\end{align*}
Since $q\equiv5\pmod{12}$, so we have $p\equiv5\pmod{12}$ and hence either $p\equiv1\pmod8$ or $p\equiv5\pmod8$. Since $p\equiv1\pmod4$, therefore there exists an element $y\in\mathbb{F}_p$ such that $y^2=-1$. If $p\equiv1\pmod8$, then 2 and $y$ are squares in $\mathbb{F}_p$. 
If $p\equiv5\pmod8$, then 2 and $y$ are nonsquares in $\mathbb{F}_p$. Hence, $\varphi(2y)=1$ for primes $p\equiv5\pmod{12}$. Thus, we have
\begin{align*}
	a_q(E_{c,d})+a_q(E_{c,-d})&=-\frac{2\varphi(c)}{q-1}\sum_{n=1,n\neq\frac{q-1}{4}}^{\frac{q-3}{2}}\overline{\omega}^{n}\left(\frac{729d^2}{16c^6}\right)(-p)^{\sum_{i=0}^{r-1}\beta_{i,n}}\prod_{i=0}^{r-1}A_{i,n}\\
	&-\varphi(d)-\varphi(-d)-\frac{2\varphi(c)}{q-1}-\frac{2\varphi(d)}{q-1}.
\end{align*}
We can easily check that $\beta_{i,\frac{q-1}{4}}=1$. Using Lemma \ref{lemma-3_2} with $t=6$ and $a=\frac{q-1}{4}$, we obtain
\begin{align*}
	\prod_{i=0}^{r-1}A_{i,\frac{q-1}{4}}=\varphi(6)\prod_{i=0}^{r-1}\frac{\Gamma_p\left(\left\langle\frac{p^i}{4}\right\rangle\right)^3\Gamma_p\left(\left\langle\frac{3p^i}{4}\right\rangle\right)^3}{\Gamma_p\left(\left\langle\frac{p^i}{2}\right\rangle\right)^4}.
\end{align*}
Using Proposition \ref{prop--1} and \eqref{eq-12} with $a=\frac{q-1}{4}$, we deduce that 
\begin{align*}
	\prod_{i=0}^{r-1}A_{i,\frac{q-1}{4}}=\varphi(6)(-1)^r\overline{\omega}^{\frac{q-1}{4}}(-1).
\end{align*}
Adding and subtracting the term under the summation for $n=0$ and $n=\frac{q-1}{4}$ and then using the fact that $\varphi(-1)=1$ yield
\begin{align*}
&a_q(E_{c,d})+a_q(E_{c,-d})\\
&=-\frac{2\varphi(c)}{q-1}\sum_{n=0}^{\frac{q-3}{2}}\overline{\omega}^{n}\left(\frac{729d^2}{16c^6}\right)(-p)^{\sum_{i=0}^{r-1}\beta_{i,n}}\prod_{i=0}^{r-1}A_{i,n}-2\varphi(d)-\frac{2\varphi(d)}{q-1}\\
&\hspace{.3cm}+\frac{2q\varphi(2d)}{q-1}\overline{\omega}^{\frac{q-1}{4}}(-1)\\
&=-\frac{\varphi(c)}{q-1}\sum_{n=0}^{\frac{q-3}{2}}\overline{\omega}^{n}\left(\frac{729d^2}{16c^6}\right)(-p)^{\sum_{i=0}^{r-1}\beta_{i,n}}\prod_{i=0}^{r-1}A_{i,n}\\
&-\frac{\varphi(c)}{q-1}\sum_{n=0}^{\frac{q-3}{2}}\overline{\omega}^{n}\left(\frac{729d^2}{16c^6}\right)(-p)^{\sum_{i=0}^{r-1}\beta_{i,n}}\prod_{i=0}^{r-1}A_{i,n}-\frac{2q\varphi(d)}{q-1}+\frac{2q\varphi(2d)}{q-1}\overline{\omega}^{\frac{q-1}{4}}(-1)\\
&=-\frac{\varphi(c)}{q-1}\sum_{n=0}^{\frac{q-3}{2}}\overline{\omega}^{n}\left(\frac{729d^2}{16c^6}\right)(-p)^{\sum_{i=0}^{r-1}\beta_{i,n}}\prod_{i=0}^{r-1}A_{i,n}\\
&-\frac{\varphi(c)}{q-1}\sum_{n=\frac{q-1}{2}}^{q-2}\overline{\omega}^{n}\left(\frac{729d^2}{16c^6}\right)(-p)^{\sum_{i=0}^{r-1}\beta_{i,n-\frac{q-1}{2}}}\prod_{i=0}^{r-1}A_{i,n-\frac{q-1}{2}}.
\end{align*}
We can easily check that $\beta_{i,n-\frac{q-1}{2}}=\beta_{i,n}$ and $A_{i,n-\frac{q-1}{2}}=A_{i,n}$. Thus, we obtain
\begin{align*}
a_q(E_{c,d})+a_q(E_{c,-d})&=	-\frac{\varphi(c)}{q-1}\sum_{n=0}^{q-2}\overline{\omega}^{n}\left(\frac{729d^2}{16c^6}\right)(-p)^{\sum_{i=0}^{r-1}\beta_{i,n}}\prod_{i=0}^{r-1}A_{i,n}\\
&=\varphi(c)\cdot{_4}G_{4}\left[\begin{array}{cccc}
0, &\frac{1}{2}, & 0, & \frac{1}{2}\vspace*{0.05cm}\\
\frac{1}{12},&\frac{5}{12},&\frac{7}{12},&\frac{11}{12}
\end{array}|\frac{729d^2}{16c^6}\right]_q.
	\end{align*}
 This completes the proof of (2).\\
Part (3):  Here, $p\equiv11\pmod{12}$ and $r=1$, i.e., $q=p$. Using Theorem \ref{thrm-1} with $d=4$ and the fact that $\varphi(-1)=-1$ for $p\equiv3\pmod 4$, we obtain the required result.
\end{proof}
\section{Proof of Theorems \ref{MT-7}, \ref{MT-8}, \ref{MT-9} and \ref{MT-10}}
We first recall a recurrence relation satisfied by the trace of the Frobenius endomorphism of elliptic curves. Let $E$ be an elliptic curve over $\mathbb{Q}$ in the Weierstrass form. If $E$ is an elliptic curve over $\mathbb{Q}$ with conductor $N_E$, 
then by modularity theorem, there exists a newform $f$ of weight 2 and level $N_E$ whose Fourier coefficients are given by the coefficients of the Hasse-Weil $L$-series $L(E,s)$ of $E$ given by
\begin{align*}
	L(E,s)=\sum_{n=1}^{\infty}\frac{a_n(E)}{n^s}.
\end{align*}
The $p$-th trace of the Frobenius endomorphism $a_p(E)$ is the $p$-th coefficient of $L(E,s)$. 
Let $\textbf{1}_{E}(p)$ is the trivial character modulo conductor $N_E$, that is, $\textbf{1}_{E}(p)$ is $1$ for primes of good reduction and $0$ for primes of bad reduction. Then, the $p^r$-th trace of Frobenius endomorphism satisfies the following recurrence relation \cite[(8.21)]{diamond}
\begin{align}\label{eqn--cc}
	a_{p^r}(E)=a_{p}(E)a_{p^{r-1}}(E)-p\textbf{1}_{E}(p)a_{p^{r-2}}(E),
\end{align}
where $r\geq2$ and $a_1(E)=1$.
\begin{proof}[Proof of Theorem \ref{MT-7}]
Let $E_{\lambda}:y^2=x(x-1)(x-\lambda)$ be an elliptic curve over $\mathbb{Q}$ such that $\lambda\neq 0,\pm1$. From \cite[(11.3)]{TM}, for $p\equiv 3\pmod 4$, we have
\begin{align*} 
	a_p(E_\frac{1}{2})=0=a_p(E_2).
\end{align*} 
Using \eqref{eqn--cc} with the fact that $a_1=1$, for $\lambda=2,\frac{1}{2}$, we obtain
\begin{align}\label{eq--23}
	a_{p^r}(E_\lambda)=\left\{\begin{array}{ll}
		0,&\hbox{if $r$ is odd;}\\
		(-p)^{\frac{r}{2}},&\hbox{if $r$ is even.}
	\end{array}\right.
\end{align}
Now, using \eqref{eq--23} and Theorem \ref{MT-3}, we obtain the desired result.
\end{proof}
\begin{proof}[Proof of Theorem \ref{MT-8}]
	Let $E_{\alpha,\frac{\alpha^3}{24}}:y^2+\alpha xy+\frac{\alpha^3}{24}y=x^3$ be an elliptic curve over $\mathbb{Q}$, where $\alpha\neq0$. 
		From \cite[(13.1)]{TM}, for $p\equiv5,11\pmod{12}$, we have
	\begin{align*} 
		a_p(E_{\alpha,\frac{\alpha^3}{24}})=0.
	\end{align*} 
	Using \eqref{eqn--cc} with the fact that $a_1=1$, we obtain
	\begin{align}\label{eq--24}
		a_{p^r}(E_{\alpha,\frac{\alpha^3}{24}})=\left\{\begin{array}{ll}
			0,&\hbox{if $r$ is odd;}\\
			(-p)^{\frac{r}{2}},&\hbox{if $r$ is even.}
		\end{array}\right.
	\end{align}
Now, using \eqref{eq--24} and Theorem \ref{MT-4}, we obtain the desired result.
\end{proof}
\begin{proof}[Proof of Theorem \ref{MT-9}]
		Let $E_{\alpha,\frac{\alpha^2}{3}}:y^2=x^3+\alpha x^2+\frac{\alpha^2}{3}x$ be an elliptic curve over $\mathbb{Q}$, where $\alpha\neq0$. From \cite[(12.1)]{TM}, for $p\equiv5,11\pmod{12}$, we have
	\begin{align*} 
		a_p(E_{\alpha,\frac{\alpha^2}{3}})=0.
	\end{align*} 
	Using \eqref{eqn--cc} with the fact that $a_1=1$, we obtain
	\begin{align}\label{eq--25}
		a_{p^r}(E_{\alpha,\frac{\alpha^2}{3}})=\left\{\begin{array}{ll}
			0,&\hbox{if $r$ is odd;}\\
			(-p)^{\frac{r}{2}},&\hbox{if $r$ is even.}
		\end{array}\right.
	\end{align}
	Now, using \eqref{eq--25} and Theorem \ref{MT-5}, we obtain the desired result.
\end{proof}
\begin{proof}[Proof of Theorem \ref{MT-10}]
Let $E_{\alpha,-\frac{2\alpha^3}{27}}:y^2=x^3+\alpha x^2-\frac{2\alpha^3}{27}$ be an elliptic curve over $\mathbb{Q}$ such that $\alpha\neq0$. From \cite[(14.1)]{TM}, for $p\equiv7, 11\pmod{12}$, we have
	\begin{align*} 
		a_p(E_{\alpha,\frac{-2\alpha^3}{27}})=0.
	\end{align*} 
	Using \eqref{eqn--cc} with the fact that $a_1=1$, we obtain
	\begin{align}\label{eq--28}
		a_{p^r}(E_{\alpha,\frac{-2\alpha^3}{27}})=\left\{\begin{array}{ll}
			0,&\hbox{if $r$ is odd;}\\
			(-p)^{\frac{r}{2}},&\hbox{if $r$ is even.}
		\end{array}\right.
	\end{align}
	Now, using \eqref{eq--28} and Theorem \ref{MT-6}, we obtain the desired result.
\end{proof}
\begin{proof}[Proof of Corollary \ref{cor1}]
Taking $\lambda=2$, $p=11$, and $r=3$ in Theorem \ref{MT-7}, we obtain
\begin{align}\label{eqn--cd}
	a_{11^3}(E_{-2})=\varphi(-1)\cdot{_4}G_4\left[\begin{array}{cccc}
		0, & \frac{1}{2}, & 0, & \frac{1}{2}\vspace*{0.1cm}\\
		\frac{1}{4}, & \frac{3}{4}, & \frac{1}{4}, & \frac{3}{4}
	\end{array}|4
	\right]_{1331}.
\end{align}
Using \texttt{SageMath}, we can easily find that $a_{11}(E_{-2})=4$ and then using \eqref{eqn--cc} and \eqref{eqn--cd}, we obtain the required result $(1)$. Taking $\alpha=2$, $p=11$ and $r=3$ in Theorem \ref{MT-8} and following similar steps as shown in the proof of $(1)$, we obtain $(2)$. Similarly, taking $\alpha=3$, $p=5$ and $r=3$ in Theorem \ref{MT-9}, we prove $(3)$. Finally, taking $\alpha=3$, $p=11$ and $r=3$ in Theorem \ref{MT-10}, we prove $(4)$.
\end{proof}

	\end{document}